\documentclass[reqno,11pt]{amsart}
\usepackage{amssymb,amsthm}
\usepackage{bm}
\usepackage{latexsym,array}
\usepackage{amsfonts,amsmath}

\newcommand{\C}{\mathbb {C}}

\newcommand{\balpha}{\boldsymbol{\alpha}}

\newcommand{\bgamma}{\boldsymbol{\gamma}}
\newcommand{\sbm}[1]{\left[\begin{smallmatrix} #1
                \end{smallmatrix}\right]}

\newcommand{\bp}{\boldsymbol{\rho}}

\newcommand{\bH}{\mathbb{H}}

\newcommand{\bv}{\bm{v}}
\newcommand{\bu}{\bm{u}}

\numberwithin{equation}{section}
\newcommand{\bbm}{\boldsymbol{\mathfrak e_{ts}}}

\newcommand{\mQ}{\boldsymbol{\mathcal Q}}

\newcommand{\bmu}{\bm{\mu}}

\newcommand{\bbl}{\boldsymbol{\ell}}
\newcommand{\bbr}{\boldsymbol{r}}
\newcommand{\bl}{\boldsymbol{\mathfrak e_\ell}}
\newcommand{\br}{\boldsymbol{\mathfrak e_r}}

\newtheorem{Pa}{Paper}[section]
\newtheorem{theorem}[Pa]{{\bf Theorem}}
\newtheorem{lemma}[Pa]{{\bf Lemma}}
\newtheorem{definition}[Pa]{{\bf Definition}}
\newtheorem{corollary}[Pa]{{\bf Corollary}}

\newtheorem{remark}[Pa]{{\bf Remark}}
\newtheorem{proposition}[Pa]{{\bf Proposition}}

\newtheorem{example}[Pa]{{\bf Example}}

\begin{document}

\title[Cyclic matrices and polynomials]
{Cyclic matrices and polynomial interpolation over division rings}
\author
{Vladimir Bolotnikov}
\address{Department of Mathematics, William and Mary, 
Williamsburg, VA 23187-8795, USA}

\begin{abstract}
As is well known, any complex cyclic matrix $A$ is similar to the unique companion matrix associated with the 
minimal polynomial of $A$. On the other hand, a cyclic matrix over a division ring $\mathbb F$ is similar to 
a companion matrix of a polynomial which is defined up to polynomial similarity. In this paper we study
more rigid canonical forms by embedding a given cyclic matrix over a division ring $\mathbb F$ into a controllable or an 
observable pair. Using the characterization of ideals in $\mathbb F[z]$ in terms of controllable and observable pairs we 
consider ideal interpolation schemes in $\mathbb F[z]$ which merge into a polynomial interpolation problems containing
both left and right interpolation conditions. 
  \end{abstract}

\maketitle

\section{Introduction}
\setcounter{equation}{0}
Given a complex matrix  $A\in\C^{n\times n}$ and a vector $\bv\in\C^n$, the sets
$$
\mathbb I_{A}:=\{p\in\C[z]: \; p(A)=0\}\quad\mbox{and}\quad  \mathbb I_{A,\bv}:=\{p\in\C[z]: \; p(A)\bv=0\}
$$
are ideals in the ring $\C[z]$ of complex polynomials; their respective (monic) generators $\bmu_A$ and 
$\mathfrak P_{A,{\bf v}}$ are called the {\em minimal polynomial} of the matrix $A$ and the {\em minimal polynomial 
of the pair} $(A,\bv)$, respectively. As $\mathbb I_{A}\subseteq \mathbb I_{A,\bv}$, it follows that $\mathfrak P_{A,{\bf v}}$
divides $\bmu_A$.

\smallskip 

A matrix $A\in\C^{n\times n}$ is called {\em cyclic} if there exists a (cyclic) vector 
$\bv\in\C^{n}$ such that ${\rm span}\{\bv,A\bv,\ldots, A^{n-1}\bv\}=\C^n$, i.e., the {\em controllability matrix}
$$
\mathfrak C_{A,\bv}=\begin{bmatrix}\bv & A\bv &\ldots & A^{n-1}\bv\end{bmatrix}
$$
is invertible (equivalently, $\deg (\mathfrak P_{A,{\bf v}})=n$). In this case, we say that the pair $(A,\bv)$ is 
{\em controllable}. 
Cyclic matrices and controllable pairs can be characterized in interpolation terms as follows.
\begin{proposition}
{\rm (1)} The matrix $A\in\C^{n\times n}$ is cyclic if and only if for any $B\in\C^{n\times n}$ commuting with $A$, 
there is an $f\in\C[z]$ such that $f(A)=B$.

\smallskip

{\rm (2)} The pair $(A,\bv)$ with $A\in\C^{n\times n}$ is controllable if and only if for   any ${\bf b}\in\C^n$, there is an
$f\in\C[z]$ such that $f(A)\bv={\bf b}$.
\label{R:1.0}
\end{proposition}
The first objective of this paper is to study possible extensions of these results as well as of some other characterizations of 
complex cyclic matrices recalled in Proposition \ref{P:1.2} below to the non-commutative
setting of a division ring $\mathbb F$. To fix notation, we let ${\bf e}_{j}$ to denote the $j$-th column of the $n\times n$ identity matrix
${\bf I}_n=\begin{bmatrix} {\bf e}_{1} & \ldots & {\bf e}_{n}\end{bmatrix}$ (occasionally writing
${\bf e}_{j,n}$ if the dimension is not clear from the context). We will use notation
\begin{equation}
\mathfrak A_n(z):=\begin{bmatrix}1 & z & \cdots & z^{n-1}\end{bmatrix}
\label{fraka}
\end{equation}
to identify a polynomial $g(z)=\mathfrak A_{n}(z){\bf g}$ with the column ${\bf g}$ of its coefficients.
In terms of this notation, we recall the {\em companion matrix} $C(f)$ of a monic polynomial $f\in\C[z]$:
\begin{equation}
C(f)=\begin{bmatrix} {\bf e}_2 & {\bf e}_3 & \ldots & {\bf e}_{n} & -{\bf f}\end{bmatrix} \; \; \mbox{if} \; \;
f(z)=z^n+\mathfrak A_n(z){\bf f}.
\label{compcomp}
\end{equation}
\begin{proposition}
Given a matrix $A\in\C^{n\times n}$, the following are equivalent:
\begin{enumerate}
\item $A$ is cyclic.
\item $\deg {\bmu}_A=n$, where $\bmu_A$ is the minimal polynomial of $A$.
\item $A$ is similar to a (unique) companion matrix (which is $C(\bmu_A)$).
\item $A$ is similar to a two-diagonal matrix 
\begin{equation}
\Gamma=\left[\begin{array}{cccc} \gamma_1&0&\ldots&0\\ 1 &\gamma_2&\ddots &\vdots \\
&\ddots&\ddots&0\\ 0& &1&\gamma_n\end{array}\right]
\label{1.1}
\end{equation}
with diagonal entries equal to zeros (roots) of the polynomial $\bmu_A$.
\item The Jordan form of $A$ contains only one Jordan block corresponding to each eigenvalue
(which is a zero of $\bmu_A$).
\end{enumerate}
If this is the case, then $\mathfrak P_{A,{\bf v}}(z)=\bmu_A(z)=
\det (z{\bf I}_n-A)$ for any cyclic vector $\bv$ of $A$.
\label{P:1.2}
\end{proposition}
In the division-ring setting, it is still true that the matrices \eqref{compcomp} and \eqref{1.1} are 
cyclic. The rest requires certain adjustments. First, the minimal polynomial $\bmu_{A}$ 
(more precisely, left and right minimal polynomials; see \eqref{2.3}) is  not similarity invariant; 
besides, simple examples show that its degree can be different from the dimension of $A$. 
The polynomial $\mathfrak P_{A,{\bf v}}$ seems to be more suitable as it is invariant under 
similarity of controllable pairs (see Definition \ref{D:1.1} below) and 
the equality $\deg (\mathfrak P_{A,{\bf v}})=n$ is equivalent to ${\bf v}$ be a (right) cyclic vector for $A$.
Although different cyclic vectors ${\bf v},{\bf v}'$ of $A$ may lead to different minimal polynomials $\mathfrak P_{A,{\bf v}}$ and
$\mathfrak P_{A,{\bf v}'}$, the latter polynomials are {\em similar}: $\mathfrak P_{A,{\bf v}}\approx\mathfrak P_{A,{\bf v}'}$ (see Section 4.1
for the precise definition). Thus, in a context that does not distinguish similar polynomials, we may write $\mathfrak P_{A}$ rather than $\mathfrak P_{A,{\bf v}}$.
Then the well-known fact that two cyclic complex matrices are similar if and only if their minimal polynomials are equal, extends 
to the non-commutative setting as follows: 
{\em cyclic matrices $A,A'\in\mathbb F^{n\times n}$ are similar if and only if $\mathfrak P_{A}\approx\mathfrak P_{A'}$}.

\smallskip

However, there is a more rigid extension in terms of controllable pairs: 
{\em two controllable pairs $(A,{\bf v})$ and $(A',{\bf v}')$  are similar if and only if
$\mathfrak P_{A,{\bf v}}=\mathfrak P_{A',{\bf v}'}$} (see Theorem \ref{T:4.5a} below).

\smallskip

In a similar manner, part (3) in Proposition \ref{P:1.2} extends to $\mathbb F$-setting in two ways: 
\begin{enumerate}
\item A cyclic matrix $A\in\mathbb F^{n\times n}$ is similar to the companion matrix $C(\mathfrak P_{A})$. 
\item For a fixed cyclic 
vector ${\bf v}$ of $A$, the controllable pair $(A,{\bf v})$ is similar to a {\em unique} pair of the form $(C(f),{\bf e}_1)$ 
(with $f=\mathfrak P_{A,{\bf v}}$). 
\end{enumerate}

A related result (Theorem \ref{T:2.10u}) describes the similarity class of a given polynomial
$p\in\mathbb F[z]$ in terms of cyclic vectors of the companion matrix $C(p)$. Similarity of a cyclic matrix to two-diagonal or 
block-diagonal matrices is discussed in Section 4.2.

\smallskip

Our second goal is to study interpolation problems of Hermite-Lagrange type in $\mathbb F[z]$. In contrast to the commutative case,
polynomials over a division ring can be evaluated on the left and on the right. Simple examples are given by left and right 
Lagrange interpolation problems where one is seeking an $f\in\mathbb F[z]$ with prescribed left (right) values at given points. 
Solution sets for homogeneous problems are right (left) ideals in $F[z]$, and under mild the assumption that the set of interpolation 
nodes is $P$-independent (see Definition \ref{D:feb1}) the problems have unique low-degree solutions. In Section 5 we consider 
more general interpolation problems with interpolation conditions given in terms of left and right tangential evaluation calculi induced
by respectively, controllable and observable pairs. In Section 6 we consider the combined (two-sided) problem containing both left and right 
interpolation conditions. The solution set of the homogeneous version of this problem is the intersection of a left and a right ideal in 
$\mathbb F[z]$, while the nonhomogeneous problem may have no solutions as well as multiple low-degree solutions. Two-sided Lagrange
problem is presented in Section 6 as an illustrative particular case.

\section{Preliminaries}
Given a division ring $\mathbb F$, let $\mathbb F[z]$ denote the ring of polynomials
in one formal variable $z$ which commutes with coefficients from $\mathbb F$. 
Since the division algorithm holds in $\mathbb F[z]$ on
either side, any ideal (left or right) in $\mathbb F[z]$ is principal. We will write $\langle f\rangle_{\bf r}$
and $\langle f\rangle_{\boldsymbol\ell}$ for the right and the left ideal generated by $f\in\mathbb F[z]$
dropping the subscript if the ideal is two-sided (i.e., left and right simultaneously). 
The intersection of two left (right) ideals is a left (right) 
ideal; the {\em least right} and {\em left common multiples} ${\bf lrcm}(f,g)$ and ${\bf llcm}(f,g)$
of two monic polynomials $f,g\in\mathbb F[z]$ are defined as generators of the respective ideals
\begin{equation}
\langle f\rangle_{\bf r}\cap\langle g\rangle_{\bf r}=
\langle{\bf lrcm}(f,g)\rangle_{\bf r}\quad\mbox{and}\quad
\langle f\rangle_{\boldsymbol\ell}\cap\langle g\rangle_{\boldsymbol\ell}=
\langle{\bf llcm}(f,g)\rangle_{\boldsymbol\ell}.
\label{2.1}
\end{equation}
If we let $Z_{\mathbb F}$ to denote the center of $\mathbb F$, then
$Z_{\mathbb F[z]}=Z_{\mathbb F}[z]$, and consequently, any ideal generated by an element of $Z_{\mathbb F}[z]$ is two-sided.
The converse is also true: the generator of any two-sided ideal of $\mathbb F[z]$ is in $Z_{\mathbb F}[z]$. Indeed, 
the left and the right (monic) generators of the ideal must be multiples of each other and therefore, they coincide.
On the other hand if $p$ is a left and right generator, then it commutes with each $\alpha\in\mathbb F$ which implies that 
its coefficients are in $Z_{\mathbb F}$; see e.g., \cite[Proposition 2.2.2]{cohn3} for details.

\subsection{Minimal polynomials} 
Any polynomial $f\in\mathbb F[z]$ can be evaluated at a matrix $A\in\mathbb F^{n\times n}$
on the left and on the right (by interpreting $\mathbb F^{n\times n}$ as an $\mathbb F$-bimodule)
as follows:
\begin{equation}
f^{\bl}(A)=\sum A^jf_j\quad\mbox{and}\quad f^{\br}(A)=\sum f_j A^j, \quad\mbox{if}\quad
f(z)=\sum f_jz^j.
\label{2.2}
\end{equation}
Since for any $f,g\in\mathbb F[z]$ and $A\in\mathbb F^{n\times n}$,
\begin{equation}
(gf)^{\bl}(A)=\sum A^j g^{\bl}(A)f_j\quad\mbox{and}\quad
(fg)^{\br}(A)=\sum f_j g^{\br}(A)A^j,
\label{2.2a}
\end{equation}
it follows that $(gf)^{\bl}(A)=0$ and $(fg)^{\br}(A)=0$ whenever 
$g^{\bl}(A)=0$ and $g^{\br}(A)=0$, respectively. Hence, the sets
\begin{equation}
\begin{array}{ll}
\mathbb I_{A,{\bf r}}&:=\{p\in\mathbb F[z]: \; p^{\bl}(A)=0\}=\langle \bmu_{A,\boldsymbol\ell}\rangle_{\bf r},\vspace{2mm}\\
\mathbb I_{A,\boldsymbol\ell}&:=\{p\in\mathbb F[z]: \; p^{\br}(A)=0\}=\langle \bmu_{A,{\bf r}}\rangle_{\boldsymbol\ell}\end{array}
\label{2.3}
\end{equation}
are respectively, a right and a left ideal in $\mathbb F[z]$; their generators $\bmu_{A,\boldsymbol\ell}$ and
$\bmu_{A,{\bf r}}$ will be called the {\em left} and the {\em right minimal polynomial} of $A$, respectively.

\smallskip

If $\mathbb I_{A,{\bf r}}$ contains a non-zero polynomial $p\in Z_{\mathbb F}[z]$ such that $p(A)=0$, then 
the set of all such polynomials form the maximal two-sided ideal contained in $\mathbb I_{A,{\bf r}}\cap \mathbb I_{A,\boldsymbol\ell}$
and its generator $\bmu_A\in Z_{\mathbb F}[z]$ is called the {\em minimal (central) polynomial of $A$}.
\begin{remark}
{\em If $A=\alpha\in\mathbb F$, then evaluations \eqref{2.2} amount to left and right ``point" evaluations of $f$ at $\alpha$:
\begin{equation}
f^{\bl}(\alpha)=\sum \alpha^jf_j\quad\mbox{and}\quad f^{\br}(\alpha)=\sum f_j \alpha^j.
\label{2.4}
\end{equation}
An element $\alpha\in\mathbb F$ is called a {\em left} or {\em right} zero of a polynomial $f\in\mathbb F[z]$ if 
$f^{\bl}(\alpha)=0$ or $f^{\br}(\alpha)=0$, respectively. The ideals $\mathbb I_{\alpha,{\bf r}}$ and $\mathbb I_{\alpha,\boldsymbol\ell}$ \eqref{2.3} 
of polynomials vanishing at $\alpha$ on the left and on the right respectively, are 
generated by the linear monic polynomial $\bmu_{\alpha,\boldsymbol\ell}=\bmu_{\alpha,{\bf r}}=\bp_\alpha$, where
\begin{equation}
\bp_\alpha(z):=z-\alpha\quad (\alpha\in\mathbb F).
\label{2.5}
\end{equation}
In other words, 
$$
f^{\bl}(\alpha)=0 \; \Leftrightarrow \; f\in \langle \bp_\alpha\rangle_{\bf r}
\quad\mbox{and}\quad f^{\br}(\alpha)=0 \; \Leftrightarrow \; 
f\in \langle \bp_\alpha\rangle_{\boldsymbol\ell}.
$$
The existence of the minimal central polynomial $\bmu_{\alpha}\in Z_{\mathbb F}[z]$ is the definition of $\alpha$ being algebraic over 
$Z_{\mathbb F}$.} 
\label{R:break}
\end{remark}
Given a polynomial $f\in\mathbb F[z]$ and given a matrix $A\in\mathbb F^{n\times n}$ and vectors 
${\bf v}\in\mathbb F^{n\times 1}$ and ${\bf u}\in\mathbb F^{1\times n}$, 
one can apply evaluations \eqref{2.2} to polynomials ${\bf v}f\in\mathbb F[z]^{n\times 1}$ and $f\bu\in\mathbb F[z]^{1\times n}$ 
as follows:
\begin{equation}
({\bf v}f)^{\bl}(A)=\sum A^j{\bf v}f_j\quad\mbox{and}\quad (f\bu)^{\br}(A)=\sum f_j \bu A^j.
\label{2.6}
\end{equation}
Due to equalities 
\begin{align}
({\bf v}gf)^{\bl}(A)&=\sum A^j \cdot ({\bf v}g)^{\bl}(A)\cdot f_j=\big(({\bf v}g)^{\bl}(A)f\big)^{\bl}(A),\label{2.6ups}\\
(fg\bu )^{\br}(A)&=\sum f_j \cdot(g\bu)^{\br}(A)\cdot A^j=\big(f(g\bu)^{\br}(A)\big)^{\br}(A),\notag 
\end{align}
holding for all $f,g\in\mathbb F[z]$, the sets 
\begin{equation}
\begin{array}{ll}
\mathbb I_{A,{\bf v}}&:=\{p\in\mathbb F[z]: \; ({\bf v}p)^{\bl}(A)=0\}=\langle \mathfrak P_{A,{\bf v}}\rangle_{\bf r},\vspace{2mm}\\
\mathbb I_{\bu,A}&:=\{p\in\mathbb F[z]: \; (p\bu)^{\br}(A)=0\}=\langle \mathfrak P_{\bu,A}\rangle_{\boldsymbol\ell}
\end{array}
\label{2.7}
\end{equation}    
are respectively a right and a left ideal in $\mathbb F[z]$; their generators are called the minimal polynomials of the {\em input pair} 
$(A,\bf v)$ and of the {\em output pair} $(\bu,A)$, respectively. 

\smallskip

Straightforward calculations show that for any 
$A\in\mathbb F^{n\times n}$, $\bv\in\mathbb F^{n\times 1}$, $\bu\in\mathbb F^{1\times n}$, and $f\in\mathbb F[z]$,
\begin{equation}
\begin{array}{ll}
\bv f(z)&=(z{\bf I}_{n}-A)\cdot (L_A(\bv f))(z)+({\bf v}f)^{\bl}(A), \\[2mm]
f(z)\bu&=(R_A(f\bu))(z)\cdot (z{\bf I}_{n}-A)+(f\bu)^{\br}(A),\end{array}\label{2.8}
\end{equation}
where $({\bf v}f)^{\bl}(A)$ and $(f\bu)^{\br}(A)$ are defined as in \eqref{2.6}
and where $L_A (\bv f)$ and $R_A (f\bu)$ are vector polynomials given by
\begin{equation}
L_A(\bv f)=\sum_{j+k=0}^{\deg f-1}
A^j \bv f_{k+j+1}z^k,\quad
R_A (f\bu)=\sum_{j+k=0}^{\deg f-1}f_{k+j+1}\bu A^jz^k.
\label{2.9}
\end{equation}
\begin{remark}
{\rm Representations \eqref{2.8} are unique in the following sense: if
$$
\bv f(z)=(z{\bf I}_{n}-A)G(z)+{\bf b}\quad\mbox{for some}\quad G\in\mathbb F^{n\times 1}[z],\quad{\bf b}\in\mathbb F^{\times 1}n, 
$$
then necessarily (by comparing the coefficients in the polynomial identity above), $G=L_A(\bv f)\quad\mbox{and}\quad {\bf b}=(\bv f)^{\bl}(A)$. 
The second representation in \eqref{2.8} is unique in a similar sense.}
\label{R:0}
\end{remark}
Ideals \eqref{2.7} can be characterized in terms of evaluations \eqref{2.6ups} as follows.
\begin{proposition}
Given $A\in\mathbb F^{n\times n}$, $\bv\in\mathbb F^{n\times 1}$, $\bu\in\mathbb F^{1\times n}$ and $f\in\mathbb F[z]$, 
\begin{enumerate}
\item $({\bf v}f)^{\bl}(A)=0$ if and only if $\bv f=(z{\bf I}_{n}-A)\cdot G$ for some $G\in\mathbb F^{n\times 1}[z]$.
\item $(f\bu)^{\br}(A)=0$ if and only if $f\bu =H\cdot (z{\bf I}_{n}-A)$ for some $H\in\mathbb F^{1\times n}[z]$.
\end{enumerate}
\label{gmvec}
\end{proposition}
\begin{proof} 
The proof is the same as in the scalar-valued case \cite[Theorem 1]{gm}. The ``only if" parts follow from equalities 
\eqref{2.8}. Conversely, comparing the coefficients
in the polynomial identity
$$
\bv f(z)=\sum_{i=0}^k \bv f_kz^k=(z{\bf I}_{n}-A)\cdot \sum_{i=0}^{k-1}G_iz^i=(z{\bf I}_{n}-A)G(z)
$$
leads us to equalities 
$$
\bv f_0=-AG_0,\quad\bv f_k=G_{k-1} \; \; \mbox{and} \; \; \bv f_j=G_{j-1}-AG_j
$$
for $j=1,\ldots, k-1$, which in turn imply
$$
({\bf v}f)^{\bl}(A)=\sum_{j=0}^{k}A^j\bv f_{j}=-AG_0+\sum_{j=1}^{k-1}A^j(G_{j-1}-AG_j)+A^kG_{k-1}=0.
$$
The proof of part (2) is similar.
\end{proof}
\begin{definition}
{\rm Let us say that two input pairs $(A,\bv)$ and $(A^\prime,\bv^\prime)$ (two output pairs $(\bu, A)$ and $(\bu^{\prime}, A^\prime)$)
are {\em similar} if  $A^\prime=TAT^{-1}$ and $\bv^\prime=T\bv$ ($\bu^{\prime}=\bu T^{-1}$)
for some invertible matrix $T\in\mathbb F^{n\times n}$. }
\label{D:1.1}
\end{definition} 
\noindent
The next observation follows from Definition \ref{D:1.1} and formulas \eqref{2.6}, \eqref{2.7}.
\begin{remark}{\rm (1) If the  input pairs $(A,\bv)$ and $(A^\prime,\bv^\prime)$ are similar with the similarity matrix $T$,
then $(\bv^\prime f)^{\bl}(A^\prime)=T\cdot (\bv f)^{\bl}(A)$ and hence $\mathbb I_{A,{\bf v}}=\mathbb I_{A^\prime,\bv^\prime}$
 and $\mathfrak P_{A,{\bf v}}=\mathfrak P_{A^\prime,{\bf v}^\prime}$.

\smallskip

(2) If the output pairs $(\bu, A)$ and $(\bu^{\prime}, A^\prime)$ are similar,
then $(f\bu)^{\br}(A)=(f\bu^\prime)^{\br}(A^\prime)\cdot T$
and hence $\mathbb I_{\bu,A}=\mathbb I_{\bu^\prime,A^\prime}$ and $\mathfrak P_{{\bf u},A}=\mathfrak P_{{\bf u}^\prime,A^\prime}$.}
\label{R:1.2}
\end{remark}
\subsection{Explicit formulas for minimal polynomials.} To compute $\mathfrak P_{A,{\bf v}}$, we first find the 
least integer $d$ such that the vectors ${\bf v},A{\bf v},\ldots,A^d{\bf v}$ are (right) linearly dependent and then conclude from 
the relation
\begin{equation}
A^d{\bf v}+A^{d-1}{\bf v}b_{d-1}+\ldots+A{\bf v}b_1+{\bf v}b_0=0
\label{2.10}
\end{equation}
that 
$$
\mathfrak P_{A,{\bf v}}(z)=z^d+{\displaystyle \sum_{k=0}^{d-1}z^k b_{k}}.
$$
 The construction of $\mathfrak P_{\bu,A}$ is similar.
As for the minimal polynomials $\bmu_{A,\boldsymbol\ell}$ and $\bmu_{A,{\bf r}}$ in \eqref{2.3}, let us observe that
\begin{equation}
\bmu_{A,\boldsymbol\ell}={\bf lrcm}(\mathfrak P_{A,{\bf e}_k}: \, 1\le k\le n),\quad 
\bmu_{A,{\bf r}}={\bf llcm}(\mathfrak P_{{\bf e}^\top_k,A}: \, 1\le k\le n), 
\label{2.11}
\end{equation}
where ${\bf e}_1,\ldots,{\bf e}_n$ denote the columns in the identity matrix ${\mathbf I}_n$. Indeed, since both $0$ and $1$ belong to 
the center of $\mathbb F$, it follows from \eqref{2.2} and \eqref{2.6} that 
$$
({\bf e}_kf)^{\bl}(A)= f^{\bl}(A){\bf e}_k\quad\mbox{for}\quad k=1,\ldots,n,
$$
so that $\mathbb I_{A,{\bf r}}=\bigcap_{k=1}^n\mathbb I_{A,{\bf e}_k}$, by \eqref{2.3} and \eqref{2.7}. 
Writing the latter equality in terms of generators we get the first equality in \eqref{2.11}; the second equality follows similarly.  
We illustrate the above recipe by computing minimal polynomials of the two-diagonal matrix $\Gamma$ as in \eqref{1.1}.
\begin{proposition}
Let $\Gamma_{\bgamma}\in\mathbb F^{n\times n}$ be of the form
\begin{equation}
\Gamma=\Gamma_{\bgamma}=\big[\delta_{i,j}\gamma_j+\delta_{i-1,j}\big]_{i,j=1}^n,\quad \bgamma=(\gamma_1,\ldots,\gamma_n)\subset\mathbb F^n,
\label{1.1j}
\end{equation}
($\delta_{i,j}$ is the Kronecker symbol) and let ${\bf e}_k$ be the $k$-th column of ${\bf I}_n$. Then
\begin{equation}
\mathfrak P_{\Gamma,{\bf e}_{k}}=\bp_{\gamma_k}\bp_{\gamma_{k+1}}\cdots \bp_{\gamma_n}\quad\mbox{and}\qquad
\mathfrak P_{{\bf e}_{k}^\top,\Gamma}=\bp_{\gamma_1}\bp_{\gamma_{2}}\cdots \bp_{\gamma_k}
\label{1.1k}
\end{equation}
for all $k=1,\ldots,n$, and consequently, the left and right minimal polynomials of $\Gamma$ are given by
\begin{align}
\bmu_{\Gamma,{\boldsymbol\ell}}&={\bf lrcm}\big(\bp_{\gamma_1}\cdots \bp_{\gamma_n}, \; \bp_{\gamma_2}\cdots \bp_{\gamma_n},\ldots,
\bp_{\gamma_{n-1}}\bp_{\gamma_n}, \; \bp_{\gamma_n}\big),\label{1.11x}\\
\bmu_{\Gamma,{\bf r}}&={\bf llcm}\big(\bp_{\gamma_1}\cdots \bp_{\gamma_n}, \; \bp_{\gamma_1}\cdots \bp_{\gamma_{n-1}},\ldots,
\bp_{\gamma_{1}}\bp_{\gamma_2}, \; \bp_{\gamma_1}\big).\notag
\end{align}
\label{P:3.15}
\end{proposition}
\begin{proof} We first observe from \eqref{1.1j} that 
\begin{equation}
\Gamma {\bf e}_j={\bf e}_j\gamma_j+{\bf e}_{j+1}\quad(j=1,\ldots,n-1)\quad\mbox{and} \quad \Gamma {\bf e}_n={\bf e}_n.
\label{1.1m}
\end{equation}
Therefore, for any fixed $k\le n$, the vectors ${\bf e}_k, \Gamma{\bf e}_k,\ldots, \Gamma^{n-k}{\bf e}_k$ are right
linearly independent. By the recipe \eqref{2.10}, it suffices to find a monic polynomial $f$ with $\deg f=n-k$ subject to
condition $({\bf e}_kf)^{\bl}(\Gamma)=0$ to claim that $\mathfrak P_{\Gamma,{\bf e}_{k}}=f$. We next show that 
$f=\bp_{\gamma_k}\bp_{\gamma_{k+1}}\cdots \bp_{\gamma_n}$ is such a polynomial. To this end, we write equalities
\eqref{1.1m} in terms of the left evaluation \eqref{2.6ups} as 
$$
({\bf e}_j\bp_{\gamma_j})^{\bl}(\Gamma)={\bf e}_{j+1}\quad(j=1,\ldots,n-1)\quad\mbox{and}\quad ({\bf e}_n\bp_{\gamma_n})^{\bl}(\Gamma)=0.
$$
Upon making use of the first formula in \eqref{2.6ups} and taking into account the latter equalities for $j=k,\ldots,n$ we get
\begin{align*}
\big({\bf e}_k \bp_{\gamma_k}\bp_{\gamma_{k+1}}\cdots \bp_{\gamma_n}\big)^{\bl}(\Gamma)&=
\big(({\bf e}_k \bp_{\gamma_k})^{\bl}(\Gamma)\bp_{\gamma_{k+1}}\cdots \bp_{\gamma_n}\big)^{\bl}(\Gamma)\\
&=\big({\bf e}_{k+1}\bp_{\gamma_{k+1}}\cdots \bp_{\gamma_n}\big)^{\bl}(\Gamma)\\
&=\ldots=\big({\bf e}_n\bp_{\gamma_n}\big)^{\bl}(\Gamma)=0,
\end{align*}
which verifies the first part in \eqref{1.1k}. The second part is verified similarly via writing relations 
$$
{\bf e}_j^\top\Gamma={\bf e}_j\gamma_j+{\bf e}_{j-1}\quad(j=2,\ldots,n)\quad\mbox{and} \quad {\bf e}_1\Gamma=\gamma_1{\bf e}_1
$$
in terms of the evaluation \eqref{2.6ups} as 
$$
(\bp_{\gamma_j}{\bf e}_j^\top)^{\br}(\Gamma)={\bf e}_{j-1}\quad(j=2,\ldots,n)\quad\mbox{and} \quad (\bp_{\gamma_1}{\bf e}_1^\top)^{\br}(\Gamma)=0
$$
and then making use of the second formula in \eqref{2.6ups}. Formulas \eqref{1.11x} follow from \eqref{1.1k} by the general principle \eqref{2.11}.
\end{proof}

\subsection{Companion matrices} For a monic polynomial $p\in\mathbb F[z]$, the associated {\em left} and 
{\em right companion matrices} are defined as
\begin{equation}
C_{\boldsymbol\ell}(p)=
\begin{bmatrix} 0 & 0 &\ldots & 0 &-p_0\\
1 & 0 & \ldots& 0 &-p_1\\
0 & 1& \ldots &0 & -p_2 \\
\vdots & \vdots &\ddots & \vdots & \vdots \\
0 & 0 &\ldots &1 & -p_{n-1} \end{bmatrix}=C_{\bf r}(p)^\top,\quad p(z)=z^n+\sum_{k=0}^{n-1} p_{k}z^{k}.
\label{2.12}
\end{equation}
By \cite[Theorem A2]{salom}, two matrices  $A,B\in\mathbb F^{n\times n}$ are  similar over a unital ring $\mathbb F$ if and only if
the pencils $z{\bf I}_n-A$ and $z{\bf I}_n-B$ are {\em equivalent} over $\mathbb F^{n\times n}[z]$ (i.e., one of them can be
transformed into another by elementary row and column operations). Combining this result with the observation that 
the pencils $z{\bf I}_n-C_{\boldsymbol\ell}(p)$ and $z{\bf I}_n-C_{\bf r}(p)$ are both equivalent over $\mathbb F^{n\times n}[z]$ to the
diagonal polynomial matrix $\sbm{{\bf I}_{n-1}& 0 \\ 0 & p(z)}$, leads to the conclusion (see \cite{salom}) that
\begin{equation}
C_{\boldsymbol\ell}(p)\sim C_{\bf r}(p)\quad\mbox{for any monic}\quad p\in \mathbb F[z].
\label{1.14g}
\end{equation}
\begin{remark}
For a monic $p\in\mathbb F[z]$ and associated companion matrices $C_{\boldsymbol\ell}(p)$ and $C_{\bf r}(p)$ (as in \eqref{2.12}),
\begin{equation}
\bmu_{C_{\boldsymbol\ell}(p),\boldsymbol\ell}=\bmu_{C_{\bf r}(p),{\bf r}}
=\mathfrak P_{C_{\boldsymbol\ell}(p),{\bf e}_1}=\mathfrak P_{{\bf e}_1^\top, C_{\bf r}(p)}=p.
\label{1.14f}
\end{equation}
{\rm Indeed, recalling ${\bf e}_j$, the $j$-th column of ${\bf I}_n$ and observing that 
\begin{equation}
C_{\boldsymbol\ell}(p)^{k-1}{\bf e}_1={\bf e}_{k}\quad (k=1,\ldots,n)\quad\mbox{and}\quad
C_{\boldsymbol\ell}(p)^{n}{\bf e}_1=-\sum_{j=1}^{n}{\bf e}_jp_j,
\label{1.14m}
\end{equation}
we see that the minimal right linearly dependent set $\{C_{\boldsymbol\ell}(p)^{j}{\bf e}_1\}_{j=1}^d$ occurs for $d=n$, and 
the relation
$$
C_{\boldsymbol\ell}(p)^{n}{\bf e}_1+\sum_{j=1}^{n}C_{\boldsymbol\ell}(p)^{n-j}{\bf e}_1p_{n-j}=0
$$
means that $\mathfrak P_{C_{\boldsymbol\ell}(p),{\bf e}_1}=p$. Multiplying the latter equality by
$C_{\boldsymbol\ell}(p)^{k}$ on the left ($k=1,\ldots, n-1$) we get, on account of \eqref{1.14m},
\begin{align*}
0=C_{\boldsymbol\ell}(p)^{n}{\bf e}_k+\sum_{j=1}^{n}C_{\boldsymbol\ell}(p)^{n-j}{\bf e}_kp_{n-j}
&=\bigg(C_{\boldsymbol\ell}(p)^{n}+\sum_{j=1}^{n}C_{\boldsymbol\ell}(p)^{n-j}p_{n-j}\bigg){\bf e}_k\\
&=p^{\bl}(C_{\boldsymbol \ell}(p)){\bf e}_k,
\end{align*}
where the second equality holds since ${\bf e}_k\in Z_{\mathbb F}^{n}$. Since  
$p^{\bl}(C_{\boldsymbol \ell}(p)){\bf e}_k=0$ for $k=1,\ldots,n$, it follows that 
$p^{\bl}(C_{\boldsymbol \ell}(p))=0$, so that $p\in\langle \bmu_{C_{\boldsymbol \ell}(p),{\boldsymbol\ell}}\rangle_{\bf r}$.
Since $p=\mathfrak P_{C_{\boldsymbol\ell}(p),{\bf e}_1}$ is a divisor of $\bmu_{C_{\boldsymbol \ell}(p),{\boldsymbol\ell}}$, it follows that
$\bmu_{C_{\boldsymbol \ell}(p),{\boldsymbol\ell}}=p$. The rest of \eqref{1.14g} is verified similarly}. 
\label{P:2.9}
\end{remark}
\begin{remark}
{\rm Although in general, $\bmu_{C_{\boldsymbol\ell}(p),{\bf r}}$ and $\bmu_{C_{\bf r}(p),{\boldsymbol\ell}}$ are not equal to $p$
(quite expectedly), they are divisible by $p$ on the right and on the left, respectively. The latter follows from the general principle 
\eqref{2.11} and equalities
$$
\mathfrak P_{{\bf e}_n^\top,C_{\boldsymbol\ell}(p)}=\mathfrak P_{C_{\bf r}(p),{\bf e}_n}=p
$$
which are verified as in the proof of Proposition \eqref{P:3.15}.}
\label{P:2.9a}
\end{remark}
We next specify formulas \eqref{2.6} and \eqref{2.9} for the case where $A$ is	
 a companion matrix (we will need them in Section 6). To this end, we recall 
the backward-shift operator $R_0$ acting on $\mathbb F[z]$ by the rule
\begin{equation}
R_0: \; \sum_{j=0}^N f_j z^j\to \sum_{j=0}^{N-1} f_{j+1} z^k.
\label{2.12u}
\end{equation}
\begin{proposition}
Let $p\in\mathbb F[z]$ and $C_{\boldsymbol\ell}(p)$ be defined as in \eqref{2.12}.
Given any $f\in\mathbb F[z]$, let us divide it by $p$ on the left: 
\begin{equation}
f(z)=p(z)g(z)+b(z), \quad g\in\mathbb F[z],\quad b(z)=b_0+b_1z+\ldots+b_{n-1}z^{n-1},
\label{2.15}
\end{equation}
and let us define the polynomials
\begin{equation}
g_j=(R_0^{j}p)\cdot g+R_0^{j}b\quad\mbox{for}\quad j=1,\ldots,n.
\label{2.12x}
\end{equation}
Then $g_n=g$ and 
\begin{equation}
({\bf e}_1 f)^{\bl}(C_{\boldsymbol \ell}(p))={\bf b}:=\left[\begin{array}{c}b_0 \\ b_1 \\ \vdots \\ b_{n-1}\end{array}\right],\quad
L_{C_{\boldsymbol \ell}(p)}({\bf e}_1f)=G:=\left[\begin{array}{c} g_1 \\ g_2 \\ \vdots \\ g\end{array}\right].\label{ups}
\end{equation}
\label{P:nul}
\end{proposition}
\begin{proof} Since $p$ is monic and $\deg p=n$, we have $R_0^{n}p=1$. Since $\deg b<n$, we have $R_0^{n}b=0$. Then for $j=n$, the formula
\eqref{2.12x} gives $g_n=g$. By definitions \eqref{2.12x}, we have
\begin{align*}
zg_{j+1}(z)+p_j g(z)+b_{j+1}&=z(R_0^{j+1}p)(z)g(z)+(R_0^{j}b(z)+p_jg(z)+b_{j}\\
&=\big[z(R_0^{j+1}p)(z)+p_j\big]g(z)+(R_0^{j}b)(z)+b_j\\
&=(R_0^{j}p)(z)g(z)g+R_0^{j}b=g_j(z)
\end{align*}
for all $j=1,\ldots,n-1$. The latter $n-1$ equalities along with \eqref{2.15} can be written in the matrix form as 
$$
\left[\begin{array}{c}f(z) \\ 0 \\ \vdots \\ 0\end{array}\right]=
\begin{bmatrix} z & 0 &\ldots & 0 &p_0\\
-1 & z & \ldots& 0 &p_1\\
\vdots & \vdots &\ddots & \vdots & \vdots \\
0 & 0 &\ldots &-1 & z+p_{n-1} \end{bmatrix}\begin{bmatrix}g_1(z)\\
g_2(z)\\ \vdots \\ g_n(z)\end{bmatrix}+
\left[\begin{array}{c}b_0 \\ b_1 \\ \vdots \\ b_{n-1}\end{array}\right].
$$
Writing this identity in the form
$\; {\bf e}_1 f(z)=(z{\bf I}_n-C_{\boldsymbol \ell}(p))G(z)+{\bf b}, \; $
we arrive at both formulas in \eqref{ups}, by Remark \ref{R:0}. 
\end{proof}
The same arguments lead us to the right-sided analogues of formulas \eqref{ups} which are recorded below for
future references.
\begin{proposition}
Let $q\in\mathbb F[z]$ $(\deg q=k$) and $C_{\bf r}(p)\in\mathbb F^{k\times k}$ be defined as in \eqref{2.12}.
Given any $f\in\mathbb F[z]$, let us divide it by $q$ on the right
$$
f(z)=h(z)q(z)+d(z),\quad h\in\mathbb F[z],\quad d(z)=d_0+d_1z+\ldots +d_{n-1}z^{k-1},
$$
and let us define the polynomials 
\begin{equation}
h_j=h\cdot (R_0^{j}q)+R_0^{j}d\quad\mbox{for}\quad j=1,\ldots,n.
\label{2.12xx}
\end{equation}
Then $h_k=h$ and 
$$
(f{\bf e}^\top_1 )^{\br}(C_{\bf r}(q))={\bf d}:=\left[ d_0 \;  \ldots \;  d_{n-1}\right],\quad
L_{C_{\bf r}(q)}(f{\bf e}_1^\top)=\left[ h_1 \; \ldots \; h\right].
$$
\label{R:nur}
\end{proposition}
\section{Controllable and observable pairs}
For $A\in\mathbb F^{n\times n}$ and $\bv\in \mathbb F^{n\times 1}$, the input pair
$(A,\bv)$ is called {\em controllable}, if its {\em controllability matrix}
\begin{equation}
\mathfrak C_{A,\bv}=\begin{bmatrix}\bv & A\bv &\ldots & A^{n-1}\bv\end{bmatrix}\quad\mbox{is
invertible}.
\label{1.2x}
\end{equation}
In this case, the vector $\bv$ is called a {\em right cyclic vector} for $A$. For similar input pairs  $(A,\bv)$
and $(A^\prime,\bv^\prime)$ (as in Definition \ref{D:1.1}), we have from \eqref{1.2x}
$$
T\mathfrak C_{A,\bv}=\begin{bmatrix}\bv^\prime & TAT^{-1}\bv^\prime &\ldots & TA^{n-1}T^{-1}\bv^\prime\end{bmatrix}
=\mathfrak C_{A^\prime,\bv^\prime},
$$
from which we see that controllability is similarity-invariant.

\smallskip

The concept of controllability goes back to \cite{kalman}; the discrete time-invariant linear system
$$
x(k+1)=Ax(k)+{\bv}w(k), \quad x(0)=x_0
$$
is called controllable, if for any preassigned $x(m)\in\mathbb F^n$, there exists the input sequence $w(0),\ldots,w(m-1)\in\mathbb F$
transferring the given initial state $x(0)$ into $x(m)$. The latter turns out to be equivalent to the controllability matrix \eqref{1.2x}
be invertible.

\smallskip

The concept of observability is dual to that of controllability: given $A\in\mathbb F^{n\times n}$ and ${\bf u}\in\mathbb F^{1\times n}$,
the discrete time-invariant linear system
$$
x(k+1)=Ax(k), \quad y(k)=\bv x(k),\quad x(0)=x_0
$$
is called {\em observable} if any preassigned output sequence $y(0),\ldots,y(n-1)$ can be generated by an appropriate initial state $x(0)$.
The latter holds if and only if the observability matrix of the pair $(\bu,A)$ is invertible which we adopt as the definition
of observability:

\smallskip

Given $A\in\mathbb F^{n\times n}$ and a row vector ${\bf u}\in\mathbb F^{1\times n}$,
the output pair $(\bu,A)$ is called {\em observable} if its {\em observability matrix}
\begin{equation}
\mathfrak O_{\bu, A}=\sbm{\bu\\ \bu A \vspace{-1mm}\\ \vdots \\ \bu A^{n-1}}\quad\mbox{is
invertible}.
\label{1.2u}
\end{equation}
In this case, $\bu$ is called a {\em left cyclic vector} for $A$. For similar output pairs $(\bu,A)$
and $(\bu^\prime,A^\prime)$, we have $\mathfrak O_{\bu^\prime, A^\prime}=\mathfrak O_{\bu, A}T^{-1}$ and hence,
observability is similarity-invariant.
\begin{remark}
{\rm It follows from the definition \eqref{1.2x} that an input pair $(A,{\bv})$ with $A\in\mathbb F^{n\times n}$
is controllable if and only if the least integer $d$ for which \eqref{2.10} holds equals $d=n$, i.e., if and only if 
$\deg \mathfrak P_{A,\bv}=n$. Similarly, an output pair $({\bf u},A)$ is observable if and only if 
$\deg \mathfrak P_{{\bf u},A}=n$.}  
\label{R:3:9}
\end{remark}
Explicit formulas for minimal polynomials of controllable and observable pairs are given below.
\begin{proposition}
{\rm (1)} If the pair $(A,\bv)$ (with $\bv\in\mathbb F^n$) is controllable, then its minimal polynomial is given by the formula
\begin{equation}
\mathfrak P_{A,\bv}(z)=z^n+\sum_{k=0}^{n-1}z^kb_k=z^n-\mathfrak A_{n}(z)
\mathfrak C_{A,\bv}^{-1}A^n\bv,
\label{1.4}
\end{equation}
(where $\mathfrak A_{n}$ is given by \eqref{fraka}), and furthermore, 
$$
\mathfrak C_{A,\bv}^{-1}A\mathfrak C_{A,\bv}=C_{\boldsymbol\ell}(\mathfrak P_{A,\bv}), \quad 
\mathfrak C_{A,\bv}^{-1}{\bf v}={\bf e}_1,
$$
i.e., the pair $(A,\bv)$ is similar to the pair $(C_{\boldsymbol\ell}(\mathfrak P_{A,\bv}),{\bf e}_1)$.

\smallskip

{\rm (2)} If the pair $(\bu,A)$ (with $\bu\in\mathbb F^{1\times n}$) is observable, then
\begin{equation}
\mathfrak P_{\bu,A}(z)=z^n-\bu A^n\mathfrak O_{\bu, A}^{-1}\mathfrak A_{n}(z)^\top,
\label{1.4u}
\end{equation}
and furthermore, the pair $(\bu,A)$ is similar to the pair $({\bf e}_1^\top, C_{\bf r}(\mathfrak P_{\bu,A}))$, as
$$
\mathfrak O_{\bu, A}A\mathfrak O_{\bu,A}^{-1}=C_{\bf r}(\mathfrak P_{\bu,A})\quad\mbox{and}\quad
{\bf u}\mathfrak O_{\bu, A}^{-1}={\bf e}_1^\top.
$$
\label{P:2.15}
\end{proposition}
\begin{proof}
If the pair $(A,{\bf v})$ is controllable, equality \eqref{2.10} holds with $d=n$ and the coefficients $b_k$
are defined by the formula
\begin{equation}
\sbm{b_0 \\ \vdots \\ b_{n-1}}=-\mathfrak C_{A,\bv}^{-1}A^n\bv.
\label{1.4d}
\end{equation}
The formula \eqref{1.4} is now immediate. By \eqref{1.2x}, $\mathfrak C_{A,\bv}^{-1}{\bf v}={\bf e}_1$ and furthermore,
\begin{align}
\mathfrak C_{A,\bv}^{-1}A\mathfrak C_{A,\bv}&=\mathfrak C_{A,\bv}^{-1}\begin{bmatrix}A\bv & \ldots &
A^{n-1}\bv & A^{n}\bv\end{bmatrix}\notag\\
&=\begin{bmatrix}{\bf e}_2 & \ldots & {\bf e}_{n-1} & \mathfrak C_{A,\bv}^{-1}A^{n}\bv\end{bmatrix}=
C_{\boldsymbol\ell}(\mathfrak P_{A,\bv})
\label{1.3}
\end{align}
where $C_{\boldsymbol\ell}(\mathfrak P_{A,\bv})$ is the left companion matrix of the polynomial \eqref{1.4}.
The second part is verified similarly.
\end{proof}
By Remark \ref{R:1.2}, similar pairs have the same minimal polynomial. For 
controllable or observable pairs, we have the converse.
\begin{theorem}
Controllable pairs $(A,\bv)$ and $(A^\prime,\bv^\prime)$ (observable pairs $(\bu,A)$ and $(\bu^\prime,A^\prime)$) are similar
in the sense of Definition \ref{D:1.1} if and only if their minimal polynomials
$\mathfrak P_{A,\bv}$ and $\mathfrak P_{A^\prime,\bv^\prime}$ ($\mathfrak P_{\bu,A}$ and $\mathfrak P_{\bu^\prime,A^\prime}$)
are equal.
\label{T:4.5a}
\end{theorem}
\begin{proof}
The ``only if" part is contained in Remark \ref{R:1.2}. To justify the ``only if part",
let us assume that the minimal polynomials $\mathfrak P_{A,\bv}$ and $\mathfrak P_{A^\prime,\bv^\prime}$ of
two controllable pairs $(A,\bv)$ and $(A^\prime,\bv^\prime)$ are equal. Then the matrices $A$ and $A^\prime$
have the same dimensions (equal to $\deg \mathfrak P_{A,\bv}=\deg \mathfrak P_{A^\prime,\bv^\prime}=n$).

\smallskip

The controllability matrices ${\mathfrak C}_{A,\bv}$ and ${\mathfrak C}_{A^\prime,\bv^\prime}$
are both invertible, and we may let $T:={\mathfrak C}_{A^\prime,\bv^\prime}{\mathfrak C}_{A,\bv}^{-1}$.
We next show that
\begin{equation}
{A}^{\prime j}\bv^\prime=TA^j\bv\quad\mbox{for}\quad j=0,\ldots,n.
\label{4.21}
\end{equation}
Indeed, comparing the corresponding columns in the matrix equality
$$
T{\mathfrak C}_{A,\bv}={\mathfrak C}_{A^\prime,\bv^\prime}
$$
gives equalities \eqref{4.21} for $j=0,\ldots, n-1$.
From the formula \eqref{1.4d} for coefficients of $\mathfrak P_{A,\bv}$ and from the similar formula for
coefficients of $\mathfrak P_{A^\prime,\bv^\prime}$ we conclude (since
$\mathfrak P_{A,\bv}=\mathfrak P_{A^\prime,\bv^\prime}$) that
${\mathfrak C}_{A,\bv}^{-1}A^n\bv={\mathfrak C}_{A^\prime,\bv^\prime}^{-1}A^{\prime n}\bv^\prime$, which can be written
equivalently as  $A^{\prime n}\bv^\prime=TA^n\bv$ thus justifying equality \eqref{4.21} for $j=n$.
Letting $j=0$ in \eqref{4.21} gives  $\bv^\prime=T\bv$, while all other equalities in \eqref{4.21} imply
$$
A^\prime{\mathfrak C}_{A^\prime,\bv^\prime}=
\begin{bmatrix}A^\prime\bv^\prime &\ldots & A^{\prime n}\bv^\prime\end{bmatrix}=
\begin{bmatrix}TA\bv&\ldots & TA^{n}\bv\end{bmatrix}=TA{\mathfrak C}_{A,\bv},
$$
which is the same as $A^\prime=TAT^{-1}$, by the definition of $T$. Thus, the pairs
$(A,\bv)$ and $(A^\prime,\bv^\prime)$ are similar. The statement concerning observable pairs follows by similar arguments.
\end{proof}

\subsection{${\mathbb I}_{A,\bv}$ and ${\mathbb I}_{\bu,A}$ as generic right and left ideals} As a consequence of Theorem \ref{T:4.5a}, 
it follows that any right or left ideal in $\mathbb F[z]$ is necessarily of the form \eqref{2.7}, and that under
controllability/observability assumption the representing pair $(A,\bv)$ or $(\bu,A)$ is unique up to similarity.
\begin{proposition}
(1) Any right ideal ${\mathbb I}\subset \mathbb F[z]$ is of the form
$$
{\mathbb I}={\mathbb I}_{A,\bv}:=
\{p\in\mathbb F[z]: \; ({\bf v}p)^{\bl}(A)=0\}=\langle \mathfrak P_{A,{\bf v}}\rangle_{\bf r}
$$
for some controllable pair $(A,\bv)$. Controllable pairs $(A,\bv)$ and $(A^\prime,\bv^\prime)$
define the same ideal ${\mathbb I}_{A,\bv}={\mathbb I}_{A^\prime,\bv^\prime}$ if and only if they are similar.

\smallskip

(2) Any left ideal ${\mathbb I}\subset \mathbb F[z]$ is of the form
$$
{\mathbb I}=\mathbb I_{\bu,A}:=\{p\in\mathbb F[z]: \; (p\bu)^{\br}(A)=0\}=\langle \mathfrak P_{\bu,A}\rangle_{\boldsymbol\ell}
$$
for some observable pair $(\bu,A)$. Observable pairs $(\bu,A)$ and $(\bu^\prime,A^\prime)$
define the same ideal ${\mathbb I}_{\bu,A}={\mathbb I}_{\bu^\prime,A^\prime}$ if and only if they are similar.
\label{C:4.5}
\end{proposition}
\begin{proof}
Given a right ideal $\mathbb I\subset \mathbb F[z]$, let $f$ be its generator.
The pair  $(C_{\boldsymbol\ell}(f),{\bf e}_1)$ is controllable and since
$\mathfrak P_{C_{\boldsymbol\ell}(f),{\bf e}_1}=f$ (by \eqref{1.14f}), we have
$$
\mathbb I=\langle f\rangle_{\bf r}=\langle \mathfrak P_{C_{\boldsymbol\ell}(f),{\bf e}_1}\rangle_{\bf r}=
{\mathbb I}_{C_{\boldsymbol\ell}(f),{\bf e}_1}.
$$
By Theorem \ref{T:4.5a}, two controllable pairs $(A,\bv)$ and $(A^\prime,\bv^\prime)$ are similar
if and only if their minimal polynomials are equal, i.e.,
$$
{\mathbb I}_{A,\bv}=\langle\mathfrak P_{A,\bv}\rangle_{\bf r}=\langle\mathfrak P_{A^\prime,\bv^\prime}\rangle_{\bf r}=
{\mathbb I}_{A^\prime,\bv^\prime}.
$$
This completes the proof of part (1). Part (2) follows similarly.
\end{proof}
\begin{remark}
Any two-sided ideal ${\mathbb I}\subset \mathbb F[z]$ is of the form ${\mathbb I}={\mathbb I}_{A,\bv}$
for a unique (up to similarity) controllable pair $(A,\bv)$ such that $\mathfrak C_{A,\bv}^{-1}A^n\bv\in Z_{\mathbb F}^{n\times 1}$.

\smallskip

Alternatively, any two-sided ideal ${\mathbb I}\subset \mathbb F[z]$ is of the form ${\mathbb I}=\mathbb I_{\bu,A}$
for a unique (up to similarity) observable pair $(\bu,A)$ such that $\bu A^n\mathfrak O_{\bu, A}^{-1}\in Z_{\mathbb F}^{1\times n}$.

\smallskip

{\rm By Propositions \ref{P:2.15} and \ref{C:4.5}, the latter statement asserts that
any two-sided ideal is a left or right ideal generated by a polynomial in $Z_{\mathbb F}[z]$, which is obviously true.}
\label{C:4.5a}
\end{remark}
We record several concrete examples of controllable and observable pairs that have already appeared above.
\begin{example}
{\rm For any monic $p\in\mathbb F[z]$ of degree $n$, the pairs $(C_{\boldsymbol\ell}(p),{\bf e}_{1})$ and $(C_{\bf r}(p),{\bf e}_n)$
are controllable, while the pairs $({\bf e}_{1}, C_{\bf r}(p)$ and $({\bf e}_{n}, C_{\boldsymbol\ell}(p)$ are observable
(since their minimal polynomials equal $p$ and $\deg p=n$)}.  
\label{E:3.2}
\end{example}
\begin{example}
{\rm If $\Gamma_{\bgamma}\in\mathbb F^{n\times n}$ is of the form \eqref{1.1j}, then the pair $(\Gamma_{\bgamma},{\bf e}_{1})$ is controllable, 
while the pair $({\bf e}_{n}^\top, \Gamma_{\bgamma})$ is observable (since their minimal polynomials are of degree $n$, by formulas \eqref{1.1k}
for $k=1$ and $k=n$, respectively).}
\label{E:3.3}
\end{example}
For the next example, we recall the notion of polynomial independence ($P$-independence) introduced in \cite{lam1}; 
see also \cite{ll1,lamler1,ll2}. 
\begin{definition}
{\rm A set $\{\alpha_1,\ldots,\alpha_n\}\subset \mathbb F$ is called {\em left (right) $P$-independent}
if the monic linear polynomials $\bp_{\alpha_1},\ldots, \bp_{\alpha_n}$ are left (right) coprime.}
\label{D:feb1}
\end{definition}
\begin{proposition}
Let $A\in\mathbb F^{n\times n}$ be diagonal and let ${\bf v}={\bf e}_1+\ldots+{\bf e}_n$:
\begin{equation}
A=\begin{bmatrix}\alpha_1 & & 0\vspace{-1mm} \\ &\ddots & \\ 0 & & \alpha_n\end{bmatrix}\quad\mbox{and}\quad
{\bf v}=\begin{bmatrix}1 \vspace{-1mm}\\ \vdots \\ 1\end{bmatrix}.
\label{4.9}
\end{equation}
The pair $(A,{\bf v})$ is controllable (the pair $(\bv^\top, A)$ is observable) 
if and only if the set $\{\alpha_1,\ldots,\alpha_n\}$ is left (right) $P$-independent.
\label{P:br4}
\end{proposition}
\begin{proof}
It follows from \eqref{2.2} and \eqref{2.6}, that for every polynomial $f\in\mathbb F[z]$,
\begin{equation}
f^{\bl}(A)=\begin{bmatrix}f^{\bl}(\alpha_1) & & 0 \vspace{-1mm} \\ &\ddots & \\ 0 & & f^{\bl}(\alpha_n)\end{bmatrix},\quad
({\bf v}f)^{\bl}(A)=\begin{bmatrix} f^{\bl}(\alpha_1)\vspace{-1mm} \\ \vdots \\ f^{\bl}(\alpha_n)\end{bmatrix},
\label{4.10}
\end{equation}
and hence, the ideals $\mathbb I_{A,{\bf r}}=\mathbb I_{A,{\bf v}}$ consist of all polynomials that vanish on the left at 
$\alpha_1,\ldots,\alpha_n$. Thus,
$\mathbb I_{A,{\bf r}}=\mathbb I_{A,{\bf v}}={\displaystyle\bigcap_{j=1}^n\langle \bp_{\alpha_i}\rangle_{\bf r}}$ and subsequently, 
\begin{equation}
\bmu_{A,\boldsymbol\ell}=\mathfrak P_{A,{\bf v}}={\bf lrcm}(\bp_{\alpha_1}, \ldots,\bp_{\alpha_n}).
\label{4.10a}
\end{equation}
One can see from the definition \eqref{1.2x} that the controllability matrix of the pair \eqref{4.9} is equal to the 
left Vandermonde matrix 
$$
\mathfrak C_{A,{\bf v}}=V_{\boldsymbol\ell}(\balpha):=\left[\alpha_i^{j-1}\right]_{i,j=1}^n,\quad
\balpha:=(\alpha_1,\ldots,\alpha_n).
$$
This matrix is invertible (i.e., the pair $(A,{\bf v})$ is controllable) if and only if $\deg \mathfrak P_{A,{\bf v}}=n$
(we recall that $\deg \mathfrak P_{A,{\bf v}}$ equals the maximal number of leftmost right linearly independent columns in 
$\mathfrak C_{A,{\bf v}}$). Due to \eqref{4.10a}, $\deg \mathfrak P_{A,{\bf v}}=n$ if and only if the polynomials 
$\bp_{\alpha_1}, \ldots,\bp_{\alpha_n}$ are left coprime, i.e., the set $\{\alpha_1,\ldots,\alpha_n\}$ is left $P$-independent.
The statement concerning the pair $(\bv^\top, A)$ is justified similarly. 
\end{proof}
The polynomials that are (left or right) minimal polynomials of an algebraic set in a division ring 
are called {\em Wedderburn polynomials}; we refer to \cite{lam1,ll1,ll2,llo} for the thorough account on the subject.
Since Wedderburn polynomials can be characterized as least common multiples of coprime monic linear polynomials,
the formula \eqref{4.10a} asserts that the pair $(A,{\bf v})$ is controllable if and only if its minimal polynomial 
$\mathfrak P_{A,{\bf v}}$ is a Wedderburn polynomial.

\smallskip

In the next proposition we will use the minimal polynomial of a controllable pair to get a closed (and fairly explicit)
formula for the least right common multiple of several given left-coprime polynomials. 
 
\begin{proposition}
Given left-coprime monic $f_1,\ldots,f_k\in\mathbb F[z]$, let 
\begin{equation}
C_{\boldsymbol\ell}(f_1,\ldots,f_k):=\begin{bmatrix}C_{\boldsymbol\ell}(f_1) & & 0 \\ &\ddots & \\ 0&& C_{\boldsymbol\ell}(f_k)\end{bmatrix},
\quad E=\begin{bmatrix}{\bf e}_{1,n_1}\\ \vdots \\ {\bf e}_{1,n_k}\end{bmatrix}, 
\label{4.10b}
\end{equation}
where $n_j=\deg f_j$, and  let $n:=n_1+\ldots+n_k$. 
Then the pair $(C_{\boldsymbol\ell}(f_1,\ldots,f_k),E)$ is controllable and its minimal polynomial is given by
\begin{align}
\mathfrak P_{C_{\boldsymbol\ell}(f_1,\ldots,f_k),E}&={\bf lrcm}(f_1,\ldots,f_k)\label{4.10c}\\
&=z^n-\mathfrak A_{n}(z)
\mathfrak C_{C_{\boldsymbol\ell}(f_1,\ldots,f_k),E}^{-1}\begin{bmatrix}C_{\boldsymbol\ell}(f_1)^n {\bf e}_{1,n_1}\vspace{-1mm} \\ \vdots \\
C_{\boldsymbol\ell}(f_k)^n {\bf e}_{1,n_k}\end{bmatrix}.\notag
\end{align}
\label{P:2.89}
\end{proposition}
\begin{proof}
By the block-diagonal structure of $C_{\boldsymbol\ell}(f_1,\ldots,f_k)$ and due to \eqref{1.14f},
$$
\mathfrak P_{C_{\boldsymbol\ell}(f_1,\ldots,f_k),E}={\bf lrcm}(\mathfrak P_{C_{\boldsymbol\ell}(f_1),{\bf e}_{1,n_1}},\ldots,
\mathfrak P_{C_{\boldsymbol\ell}(f_k),{\bf e}_{1,n_k}})= {\bf lrcm}(f_1,\ldots,f_k). 
$$
Since $f_1,\ldots,f_k$ are left coprime, 
$$
\deg ({\bf lrcm}(f_1,\ldots,f_k))=\deg f_1+\ldots+\deg f_k=n.
$$
Hence, the two last formulas imply $\deg \mathfrak P_{C_{\boldsymbol\ell}(f_1,\ldots,f_k),E}=n$. Therefore, 
the matrix $\mathfrak C_{C_{\boldsymbol\ell}(f_1,\ldots,f_k),E}$ is invertible
(i.e., the pair $(C_{\boldsymbol\ell}(f_1,\ldots,f_k),E)$ is controllable), and we can apply formula \eqref{1.4} to complete 
the proof of \eqref{4.10c}. 
\end{proof}
\noindent
\section{Cyclic matrices and similarity reduction} A matrix $A$ over a division ring $\mathbb F$ is called {\em cyclic} if it admits a (left or right) cyclic vector, i.e., 
if it can be embedded into a controllable or an observable pair. Alternatively, cyclic matrices can be defined 
as the ones similar to companion matrices or as the matrices having one non-constant invariant factor.
All these equivalent definitions are recorded below.
\begin{theorem}
Given a matrix $A\in\mathbb F^{n\times n}$, the following are equivalent:
\begin{enumerate}
\item There exists $\bv\in\mathbb F^{n\times 1}$ such that the pair $(A,\bv)$ is controllable.
\item There exists ${\bf u}\in\mathbb F^{1\times n}$ such that the pair $(\bu, A)$ is observable.
\item $A$ is similar to a (left or right) companion matrix.
\item The pencil $z{\bf I}_n-A$ is equivalent to a polynomial matrix $\sbm{{\bf I}_{n-1}&0\\ 0& h(z)}$.
\end{enumerate}
If this is the case (i.e., if $A$ is cyclic), then
\begin{itemize}
\item[(a)] $A\sim C_{\boldsymbol\ell}(\mathfrak P_{A,\bv})\sim C_{\bf r}(\mathfrak P_{\bu,A})$ for any cyclic vectors $\bv,\bu$ of $A$.
\item[(b)] If $A=TC_{\boldsymbol\ell}(f)T^{-1}=SC_{\bf r}(g)S^{-1}$ for some $f,g\in\mathbb F[z]$,
then the vectors $\bv=T{\bf e}_1$ and $\bu={\bf e}_1^\top S^{-1}$ are cyclic for $A$ and furthermore,
$$f=\mathfrak P_{A,\bv}\quad\mbox{and}\quad g=\mathfrak P_{\bu,A}.
$$
\item[(c)] The invariant factor $h$ of $A$ from part (4) is necessarily of the form $h=\mathfrak P_{A,{\bf w}}$ for some 
cyclic vector ${\bf w}$ of $A$. 
\end{itemize}
\label{P:ex}
\end{theorem}
\begin{proof} Each one of the properties (1)-(4) is similarity invariant.  By \eqref{1.3} and \eqref{1.2u}, $A$ admits 
a right (left) cyclic vector if and only if it is similar to a left (right) companion matrix. By \eqref{1.14g}, we now conclude that
the statements (1), (2), (3) are equivalent. Since (4) holds (with $h=f$) for any companion matrix $C_{\boldsymbol\ell}(f)$,
the equivalence $(3)\Leftrightarrow(4)$ follows. 

\smallskip

The statement (a) follows from Proposition \ref{P:2.15}. If $A=TC_{\boldsymbol\ell}(f)T^{-1}$ and 
$\bv=T{\bf e}_1$, then the input pairs $(A,\bv)$ and $(C_{\boldsymbol\ell}(f),{\bf e}_1)$ are similar (see
Definition \ref{D:1.1}) and therefore, $\mathfrak P_{A,\bv}=\mathfrak P_{C_{\boldsymbol\ell}(f),{\bf e}_1}=f$,
by Remark \ref{R:1.2} and due to \eqref{1.14f}. The rest of the part (b) follows from similarity of output pairs $(\bu,A)$ and
$({\bf e}_1^\top, C_{\bf r}(g))$. Finally, being an invariant factor of $A$, the polynomial $h$ in (4) is also an invariant factor
for its companion matrix $C_{\boldsymbol\ell}(h)$. Therefore, the pencils $z{\bf I}_n-A$ and $z{\bf I}_n-C_{\boldsymbol\ell}(h)$
are equivalent and hence, $A\sim C_{\boldsymbol\ell}(h)$. Now part (c) follows from (b).
\end{proof}

\subsection{Similar polynomials} In the contrast to the commutative 
setting of Proposition \ref{P:1.2}, similar companion matrices over a noncommutative division ring do not have to be equal
(for an example, take two similar elements $\alpha\sim\alpha'$ (i.e., $\alpha \beta=\beta\alpha'$ for some $\beta\neq 0$)
and consider the companion matrices $C_{\boldsymbol\ell}(\bp_\alpha)=\alpha$ and $C_{\boldsymbol\ell}(\bp_{\alpha'})=\alpha'$). 
The polynomials generating similar companion matrices are called {\em similar}; in notation: $f\approx g$.
\begin{proposition}
For polynomials $f,g\in\mathbb F[z]$, the following are equivalent:
\begin{enumerate}
\item ${\bf lrcm}(f,h)=hg$ for some $h\in\mathbb F[z]$ such that $f, h$ are left coprime. 
\item ${\bf llcm}(g,p)=fp$ for some $p\in\mathbb F[z]$ such that $p, g$ are right coprime.
\item $fp=hg$ for some $h,p\in\mathbb F[z]$ such that $f, h$ are left coprime and $p, g$ are right coprime.
\item $C_{\boldsymbol\ell}(f)\sim C_{\boldsymbol\ell}(g)$.
\end{enumerate}
\label{P:2.2}
\end{proposition}
Property (1) is the original definition of similar polynomials that appeared in \cite{ore}. 
The equivalence $(1)\Leftrightarrow(2)$ was shown in \cite[Theorem 1.18]{ore}. 
If $f,g$ satisfy (1), then $fp=hg$ for some $p\in\bH[z]$ which is necessarily right coprime 
with $g$ (for otherwise, $hg$ wouldn't be the {\em least} right common multiple of $f$ and $h$).
On the other hand, if (3) is in force, then (1) holds with the same $h$ 
(for otherwise, the polynomial ${\bf lrcm}(f,h)=f\widetilde p=h\widetilde g$ would be a proper left divisor
of $fp=hg$ implying that $p=\widetilde p q$ and $g=\widetilde g q$ for some non-constant $q\in\bH[z]$
contradicting the right coprimeness of $p$ and $g$). Property (3) appears as the definition of polynomial similarity  
in \cite{fit,jacob,cohn} in terms of isomorphic cyclic modules. For the equivalence 
$(3)\Leftrightarrow(4)$, see e.g., \cite[Theorem 4.9]{llo}.

\smallskip

As a consequence of Theorem \ref{P:ex}, we have the following relaxed version of Theorem \ref{T:4.5a} (when similarity is imposed on
state space matrices rather then on input or output pairs).
\begin{proposition}
(1) The minimal polynomials  $\mathfrak P_{A,\bv}$ and $\mathfrak P_{A^\prime,\bv^\prime}$ of controllable pairs 
$(A,\bv)$ and $(A^\prime,\bv^\prime)$ are similar if and only if $A\sim A^\prime$.

(2) The minimal polynomials $\mathfrak P_{\bu,A}$ and $\mathfrak P_{\bu^\prime,A^\prime}$
of observable pairs $(\bu,A)$ and $(\bu^\prime,A^\prime)$ are similar if and only if $A\sim A^\prime$.
\label{R:2.4}
\end{proposition}
\noindent
Indeed, by Proposition \ref{P:2.15}, $A\sim C_{\boldsymbol\ell}(\mathfrak P_{A,\bv})$ and $A^\prime
\sim C_{\boldsymbol\ell}(\mathfrak P_{A^\prime,\bv^\prime})$. By Theorem \ref{P:ex}, we therefore have
$$
A\sim A^\prime \; \Leftrightarrow \, C_{\boldsymbol\ell}(\mathfrak P_{A,\bv})\sim
C_{\boldsymbol\ell}(\mathfrak P_{A^\prime,\bv^\prime}) \, \Leftrightarrow \, \mathfrak P_{A,\bv}\approx\mathfrak P_{A^\prime,\bv^\prime}.
$$
The second statement follows similarly, due to \eqref{1.14g}. 

\smallskip

Upon combining Theorem \ref{P:ex} and Theorem \ref{T:4.5a}, we arrive at the following parametrization of the similarity
class of a given polynomial.
\begin{theorem}
Given a monic $p\in\mathbb F[z]$, the formula
$$
\varphi: \, \bv\mapsto \mathfrak P_{C_{\boldsymbol\ell}(p),\bv}
$$
establishes a map from the set of all cyclic vectors of the companion matrix $C_{\boldsymbol\ell}(p)$ onto
the similarity class of $p$. Moreover, $\varphi(\bv)=\varphi(\bv^\prime)$ if and only if there exists an invertible
$T\in\mathbb F^{n\times n}$ such that 
\begin{equation}
TC_{\boldsymbol\ell}(p)=C_{\boldsymbol\ell}(p)T\quad\mbox{and}\quad \bv^\prime=T\bv.
\label{comm}
\end{equation}
\label{T:2.10u}
\end{theorem}
\begin{proof}
The vector ${\bf e}_1$ is cyclic for $C_{\boldsymbol\ell}(p)$ and $\mathfrak P_{C_{\boldsymbol\ell}(p),{\bf e}_1}=p$.
By Proposition \ref{R:2.4} (part (1) with $A=A^\prime=C_{\boldsymbol\ell}(p)$), 
$$
\mathfrak P_{C_{\boldsymbol\ell}(p),\bv}\approx\mathfrak P_{C_{\boldsymbol\ell}(p),{\bf e}_1}=p
$$ 
for any cyclic vector ${\bf v}$ of $C_{\boldsymbol\ell}(p)$. Conversely, if $g\approx p$, then 
$C_{\boldsymbol\ell}(p)\sim C_{\boldsymbol\ell}(g)$, i.e., $C_{\boldsymbol\ell}(p)=T^{-1}C_{\boldsymbol\ell}(g)T$
for some invertible $T\in\mathbb F^{n\times n}$. If we let ${\bf v}=T^{-1}{\bf e}_1$, then controllable pairs 
$(C_{\boldsymbol\ell}(p),{\bf v})$ and $(C_{\boldsymbol\ell}(g),{\bf e}_1)$ will be similar and hence,
$$
\mathfrak P_{C_{\boldsymbol\ell}(p),\bv}=\mathfrak P_{C_{\boldsymbol\ell}(g),{\bf e}_1}=g,
$$
by Theorem \ref{T:4.5a}. Therefore, the map $\varphi$ is onto. Again due to Theorem \ref{T:4.5a}, 
$\mathfrak P_{C_{\boldsymbol\ell}(p),\bv}=\mathfrak P_{C_{\boldsymbol\ell}(p),\bv^\prime}$ if and only if the pairs
$(C_{\boldsymbol\ell}(p),{\bf v})$ and $C_{\boldsymbol\ell}(p),{\bf v^\prime}$ are similar, which is equivalent to 
relations \eqref{comm}.
\end{proof}

\subsection{Similarity reduction} A cyclic matrix $A\in\mathbb F^{n\times n}$ does not have to be similar to a two-diagonal matrix.
To address Proposition \ref{P:1.2} (part (4)) in the non-commutative setting, we put it in the following form: {\em if the companion 
matrix $C(f)$ ($f\in\C[z]$) is similar to a matrix $\Gamma$ \eqref{1.1}, then necessarily $f=\bp_{\gamma_1}\cdots\bp_{\gamma_n}$}.
Two noncommutative extensions of the latter statement are given in Propositions \ref{C:3.76} and \ref{T:1.16} below.
\begin{proposition}
Let $f\in\mathbb F[z]$ be a monic polynomial of degree $n$. Let $\Gamma_{\bgamma}$, $A$ and $C_{\boldsymbol\ell}(f_1,\ldots,f_k)$
be $n\times n$ matrices given by \eqref{1.1j}, \eqref{4.9} and \eqref{4.10b}, respectively. Then
\begin{enumerate}
\item $C_{\boldsymbol\ell}(f)\sim\Gamma_{\bgamma}$ if and only if $f\approx \bp_{\gamma_1}\bp_{\gamma_2}\cdots \bp_{\gamma_n}$.
\item $C_{\boldsymbol\ell}(f)\sim A={\rm diag}(\alpha_1,\ldots,\alpha_n)$ if and only if 
$f\approx {\bf lrcm}(\bp_{\alpha_1},\ldots,\bp_{\alpha_n})$.
\item $C_{\boldsymbol\ell}(f)\sim {\rm diag}(C_{\boldsymbol\ell}(f_1),\ldots,
C_{\boldsymbol\ell}(f_k))$ if and only if $f\approx {\bf lrcm}(f_1,\ldots,f_k)$.
\end{enumerate}
\label{C:3.76}
\end{proposition}
All three statements are known in the more general setting of skew polynomials \cite[Section 5]{llo}. In the present context, they 
follow from Proposition \ref{R:2.4} and formulas \eqref{1.1k}, \eqref{4.10a}, \eqref{4.10c} and \eqref{1.14f}. 

\smallskip
\noindent
The main point in part (1) is: {\em $C_{\boldsymbol\ell}(f)$ is similar to a two-diagonal matrix of the form \eqref{1.1j} 
if and only if $f$ splits into the product of linear factors}. 
Part (2) says that {\em $C_{\boldsymbol\ell}(f)$ is similar to a diagonal matrix if and only if $f$ is a Wedderburn polynomial}.
Since $\alpha_i=C_{\boldsymbol\ell}(\bp_{\alpha_i})$, part (2) can be interpreted as the extremal particular case of part (3)
when all diagonal blocks in the matrix $C_{\boldsymbol\ell}(f_1,\ldots,f_k)$ are scalars.
The opposite extremal case is the one where $f$ cannot be represented as the {\bf lrcm} of its proper left divisors, or equivalently,
the ideal $\langle f\rangle_{\bf r}$ is irreducible in the sense that it is not contained into two distinct proper right ideals in 
$\mathbb F[z]$. Following Ore \cite{ore} we will call such polynomials {\em indecomposable}. By \cite[Theorem 13, Part II]{ore},
any polynomial $f\in\mathbb F[z]$ admits a representation 
\begin{equation}
f={\bf lrcm}(f_1,\ldots,f_k)
\label{ore}
\end{equation}
where $f_1,\ldots,f_k$ are left coprime indecomposable polynomials, and this representation is unique up to similarity of each component.

\smallskip

We now present a more rigid version of Proposition \ref{C:3.76} dealing with the fixed controllable pair 
$(C_{\boldsymbol\ell}(f),{\bf e}_1)$ rather than the companion matrix itself.
\begin{proposition}
Let $f\in\mathbb F[z]$ be a monic polynomial of degree $n$, let $\Gamma_{\bgamma}$ and the controllable pairs
$(A,{\bv})$, $(C_{\boldsymbol\ell}(f_1,\ldots,f_k),E)$ be defined as in \eqref{1.1j}, \eqref{4.9} 
and \eqref{4.10b}, respectively. Then
\begin{enumerate}
\item $(C_{\boldsymbol\ell}(f),{\bf e}_1)\sim(\Gamma_{\bgamma}, {\bf e}_1)$, i.e.,
there exists an invertible $T$ such that
$$
TC_{\boldsymbol\ell}(f)T^{-1}=\Gamma_{\bgamma}\quad\mbox{and}\quad T{\bf e}_1={\bf e}_1
$$
if and only if $f=\bp_{\gamma_1}\cdots\bp_{\gamma_n}$.
\item $(C_{\boldsymbol\ell}(f),{\bf e}_1)\sim(A,{\bf v})$  if and only if
$f={\bf lrcm}(\bp_{\alpha_1},\ldots,\bp_{\alpha_n})$.
\item $(C_{\boldsymbol\ell}(f),{\bf e}_1)\sim (C_{\boldsymbol\ell}(f_1,\ldots,f_k), E)$ 
if and only if $f={\bf lrcm}(f_1,\ldots,f_k)$.
\end{enumerate}
\label{T:1.16}
\end{proposition}
\begin{proof} Since the pairs $(C_{\boldsymbol\ell}(f),{\bf e}_1)$, $(\Gamma_{\bgamma},{\bf e}_1)$, $(A,{\bf v})$ and 
$(C_{\boldsymbol\ell}(f_1,\ldots,f_k), E)$ are all controllable, any two of them are similar if and only if their 
minimal polynomials are equal. Since $\mathfrak P_{C_{\boldsymbol\ell}(f),{\bf e}_1}=f$ the statements follow
from the first formula in \eqref{1.1k} (for $k=1$) and formulas \eqref{4.10a} and \eqref{4.10c}. 
\end{proof}
In case $\mathbb F=\mathbb H$, the skew field quaternions, some of the previous results 
can be elaborated a bit further, due to the facts that any non-real (i.e., non-cetral) element in $\mathbb H$ is algebraic of 
degree two and that $\mathbb H$ is algebraically closed on the left and on the right and hence any polynomial $f\in\mathbb H$
splits in $\mathbb H$. In this setting, {\em any} cyclic matrix is similar to a two-diagonal matrix \eqref{1.1}, which is the exact 
analog of part (3) in Proposition \ref{P:1.2}. Part (3) in Proposition \ref{C:3.76} is worked out to get the 
Jordan form of a cyclic matrix $A$ (which necessarily contains one block corresponding to each real
eigenvalue and at most two blocks corresponding to each non-real eigenvalue), while part (3) in Proposition \ref{T:1.16} eventually 
establishes similarity of a controllable pair $(A,{\bf v})$ to the essentially unique pair $(\mathcal J,E)$ where $\mathcal J$
is the block-diagonal matrix with diagonal blocks of the form \eqref{1.1} where this time, all $\gamma_j$'s are similar to each other
and $\gamma_{j+1}\neq \overline{\gamma}_j$ (the quaternion conjugate of $\gamma_j$). We omit details.

\section{Ideal interpolation schemes}
\setcounter{equation}{0}

Characterizations of ideals of $\mathbb F[z]$ given in Proposition \ref{C:4.5} and Remark \ref{C:4.5a} 
in terms of evaluations \eqref{2.6}
based on controllable and observable pairs suggest to take yet another look at interpolation problems
in $\mathbb F[z]$. We start with ideal interpolation schemes that were proposed in \cite{bir} in an attempt to come 
up with meaningful multivariate analogues of the Lagrange-Hermite interpolation problem. 
The single-variable non-commutative version of this concept is the following: given a ring $\mathbb F$,
a finite set $\Phi=\{\phi_i\}_{i=1}^n$ of linearly independent functionals
$\varphi_i: \, \mathbb F[z]\mapsto \mathbb F$ is called a {\em right (left, two-sided) ideal interpolation scheme} 
if $\bigcap_{i=1}^n{\rm Ker} \, \varphi_i$ is a right (left, two-sided) ideal in $\mathbb F[z]$. 
Given an ideal interpolation scheme, the associated interpolation problem consists of finding all
$f\in\mathbb F[z]$ such that $\varphi_i(f)=c_i$ for preassigned $c_j\in\mathbb F$ ($i=1,\ldots,n$). 
Since the problem is linear, the answer for a left (or right) scheme is given by the respective 
formulas (which are the same if the scheme is two-sided)
\begin{equation}
f=f_0+ph\quad\mbox{or}\quad f=f_0+hp, 
\label{kol}
\end{equation}
where $p$ is the generator of the ideal $\bigcap_{i=1}^n{\rm Ker} \, \varphi_i$,
where $f_0$ is a unique particular solution to the problem with $\deg f_0<\deg p$, and where $h\in\mathbb F[z]$ is a free parameter.
The only remaining question is to find an explicit formula for $f_0$ in terms of given $c_1,\ldots,c_n$ and $p$.

\smallskip
	
By Proposition \ref{C:4.5}, any right (left) ideal interpolation scheme in $\mathbb F[z]$ can be embedded into the 
following left (right) interpolation problem with the interpolation condition given in terms of evaluations 
\eqref{2.6}. 

\medskip
\noindent
{\bf Problem} ${\bf LP}(A,\bv,{\bf b})$: {\em given a controllable pair $(A,\bv)$ with $A\in\mathbb F^{n\times n}$, 
and given ${\bf b}\in\mathbb F^{n\times 1}$, find a polynomial $f\in\mathbb F[z]$ such that}
\begin{equation}
(\bv f)^{\bl}(A)={\bf b}.
\label{4.22}
\end{equation}
\noindent
{\bf Problem} ${\bf RP}(B,\bu,{\bf d})$: {\em given an observable pair $(\bu, B)$ with $B\in\mathbb F^{k\times k}$,
and given ${\bf d}\in\mathbb F^{1\times k}$, find a polynomial $f\in\mathbb F[z]$ such that}
\begin{equation}
(f\bu)^{\br}(B)={\bf d}.
\label{4.24}
\end{equation}
The next two results specifying the parametrization formulas \eqref{kol} in terms of
interpolation data can be regarded as left and right noncommutative analogues of Proposition \ref{R:1.0} (part (2)).
\begin{theorem}
The input pair $(A,\bv)$ with $A\in\mathbb F^{n\times n}$ is controllable if and only if the problem ${\bf LP}(A,\bv,{\bf b})$
has a solution for any ${\bf b}\in\mathbb F^{n\times 1}$. In this case, all solutions to the problem  
are parametrized by the formula
\begin{equation}
f={f}_{\bbl}+\mathfrak P_{A,\bv}h,\quad\mbox{with}\quad {f}_{\bbl}(z)=\mathfrak A_{n}(z)
\mathfrak C_{A,\bv}^{-1}{\bf b},\quad h\in\mathbb F[z],
\label{4.22b}
\end{equation}
where $\mathfrak A_{n}$ is defined in \eqref{fraka}, $\mathfrak P_{A,\bv}$ is the minimal polynomial of the pair $(A,\bv)$, $f_{\bbl}$ 
is the low-degree solution, and $h$ is a free parameter.
\label{R:1.0a}
\end{theorem}
\begin{proof}
by the division algorithm, the problem ${\bf LP}(A,\bv,{\bf b})$ has a solution if and only if it has
a low-degree one. To find a polynomial $f_{\bbl}$ with $\deg f<n$  and subject to condition \eqref{4.22}, we may take it in the form
\begin{equation}
f_{\bbl}(z)=\sum_{j=0}^{n-1}f_jz^j=\mathfrak A_{n}(z)F, \qquad F=\sbm{f_0\\ \vdots \\ f_{n-1}},
\label{4.23}
\end{equation}
and then compute, upon making use of \eqref{2.6} and \eqref{1.2x},
$$
(\bv f_{\bbl})^{\bl}(A)={\bf v}f_0+A{\bf v}f_1+\ldots+A^{n-1}{\bf v}f_{n-1}={\mathfrak C}_{A,\bv}F.
$$
Thus the equation $(\bv f)^{\bl}(A)=\mathfrak C_{A,\bv}F={\bf b}$ has a solution
$F$ for any ${\bf b}\in\mathbb F$ if and only if the controllability matrix $\mathfrak C_{A,\bv}$ is invertible, i.e., 
the pair $(A,\bv)$ is controllable. In this case, $f_{\bbl}$ satisfies condition \eqref{4.22} if and only if 
$F={\mathfrak C}_{A,\bv}^{-1}{\bf b}$ which being substituted into \eqref{4.23}, gives \eqref{4.22b}.
Since the solution set of the homogeneous problem ${\bf LP}(A,\bv,0)$ is the right ideal $\langle \mathfrak P_{A,{\bf v}}\rangle_{\bf r}$,
(by Proposition \ref{C:4.5}), the parametrization formula \eqref{4.22b} follows.
\end{proof}
\noindent
The right-sided version of Theorem \ref{R:1.0a} presented below is justified similarly.
\begin{theorem}
The output pair $(\bu,B)$ with $B\in\mathbb F^{k\times k}$ is observable if and only if the problem ${\bf RP}(B,\bu,{\bf d})$
has a solution for any ${\bf d}\in\mathbb F^{1\times k}$. In this case, all solutions to the problem 
are given by the formula
\begin{equation}
f={f}_{\bbr}+h\mathfrak P_{\bu,B},\quad\mbox{with}\quad {f}_{\bbr}(z)={\bf d}\mathfrak O_{\bu,B}^{-1}\mathfrak A_{k}(z)^\top,
\quad h\in\mathbb F[z],
\label{4.22c}
\end{equation}
where $\mathfrak P_{\bu,A}$ is the minimal polynomial of the pair $(\bu,A)$, $f_{\bbr}$ is the low-degree solution, 
and $h$ is a free parameter.
\label{R:1.0c}
\end{theorem}
\begin{remark}
{\rm The problem \eqref{4.22} can be efficiently solved for {\em any} (not necessarily controllable) 
input pair $(A,{\bf v})$ as follows. Given a pair $(A,{\bf v})$,
we find the smallest integer $d$ such that the vectors ${\bf v},A{\bf v},\ldots,A^d{\bf v}$ are (right) linearly dependent and then 
construct the minimal polynomial $\mathfrak P_{A,\bv}$ (of degree $d$) as suggested in \eqref{2.10}. The problem \eqref{4.22}
has a solution if and only if the column ${\bf b}$ belongs to the right range space of $\mathfrak C_{A,\bv}$ (i.e.,  
to the right linear span of ${\bf v},A{\bf v},\ldots,A^{d-1}{\bf v}$), the {\em controllability space} of the pair 
$(A,{\bf v})$. If this is the case, we represent ${\bf b}$ as 
$$
{\bf b}=\sum_{j=0}^{d-1} A^j{\bf v}b_j,\quad\mbox{and let}\quad f_\ell(z)=\sum_{j=0}^{d-1}b_jz^j.
$$ 
It is readily seen that all polynomials $f$ subject to the interpolation condition \eqref{4.22} are parametrized by the
formula \eqref{4.22b}.}
\label{R:1.0d}
\end{remark}
In conclusion we briefly address the first statement in Proposition \ref{R:1.0}. In the complex setting,
$f(A)$ commutes with $A$, so the ``only if" part is immediate.  
As complex polynomials respect similarity, the matrix $A$ can be taken
in the canonical Jordan form, and then the commutativity relation $AB=BA$ forces $B$ to be of triangular
block Toeplitz structure. Then the Hermite-Lagrange polynomial with prescribed values
(determined by $B$ at eigenvalues of $A$ (with multiplicities) satisfies $f(A)=B$. 

\smallskip

In contrary to this case, polynomials over $\mathbb F$ do not respect similarity and besides, $f(A)$ does not have to commute with $A$.
The solvability of the interpolation problem $f^{\bl}(A)=B$ for every $B\in\mathbb F^{n\times n}$ does not seem to have much to do
with cyclicity of $A$. However, this problem falls in the left ideal interpolation scheme and its particular solution can be 
found recursively as follows. Letting $B=\begin{bmatrix}{\bf b}_1&{\bf b}_2 &\ldots & {\bf b}_n\end{bmatrix}$ we recall that ${\bf e}_j\in Z_{\mathbb F}^n$
and write the interpolation 
condition $f^{\bl}(A)=B$ equivalently as 
\begin{equation}
\big({\bf e}_jf\big)^{\bl}(A)={\bf b}_j\quad\mbox{for}\quad j=1,\ldots,n.
\label{4.22k}
\end{equation}
Applying the procedure from Remark \ref{R:1.0c}, we either conclude that the first condition in \eqref{4.22k} is inconsistent (and hence the 
problem has no solutions) or we get all polynomials subject to this conditions in the form
\begin{equation}
f=f_{{\boldsymbol\ell}}+\mathfrak P_{A,{\bf e}_1}h,\qquad h\in\mathbb F[z].
\label{4.22m}
\end{equation}
Making use of \eqref{2.6ups}, we see that a polynomial $f$ of the form \eqref{4.22m} satisfies conditions \eqref{4.22} for $j=2,\ldots,n$ if and only if
$$
{\bf b}_j=\big({\bf e}_jf\big)^{\bl}(A)=\big({\bf e}_jf_{{\boldsymbol\ell}}\big)^{\bl}(A)+\big(\big({\bf e}_j\mathfrak P_{A,{\bf e}_1}\big)^{\bl}(A))h\big)^{\bl}(A)
$$
for $j=2,\ldots,n$, which can be written in terms of the parameter $h$ as 
\begin{equation}
\big({\bf c}_jh\big)^{\bl}(A)={\bf d}_j\quad\mbox{for}\quad j=2,\ldots,n,
\label{4.22p}
\end{equation}
where 
${\bf c}_j=\big({\bf e}_j\mathfrak P_{A,{\bf e}_1}\big)^{\bl}(A)$ and  ${\bf d}_j={\bf b}_j-\big({\bf e}_jf_{{\boldsymbol\ell}}\big)^{\bl}(A)$. Thus,
either the problem \eqref{4.22k} is inconsistent or it reduces (via \eqref{4.22m} to a similar problem \eqref{4.22p} with fewer conditions.
Continuing this reduction, we either conclude that the original problem \eqref{4.22k} is inconsistent or will come up (in $n$ steps) with a parametrization 
of all its solutions.

\section{Two-sided interpolation}
\setcounter{equation}{0}

Our next goal is to consider the problem which arises by combining left and right ideal interpolation schemes. 
We will call this problem {\em two-sided} as it contains both left and right interpolation conditions.

\medskip
\noindent
{\bf Problem} ${\bf TSP}$: {\em Given a controllable pair $(A,\bv)$ and an observable pair $(\bu, B)$ (with $A\in\mathbb F^{n\times n}$,
$B\in\mathbb F^{k\times k}$), along with the 
target vectors ${\bf b}$ and ${\bf d}$, find a polynomial $f\in\mathbb F[z]$ subject to conditions
\begin{equation}
(\bv f)^{\bl}(A)={\bf b},\quad (f\bu)^{\br}(B)={\bf d},\quad \deg f<n+k.
\label{4.25y}
\end{equation}}
By Theorems \ref{R:1.0a} and \ref{R:1.0c}, the latter problem can be identified with the following one: 
{\em given polynomials $p, q, {f}_{\bbl}, {f}_{\bbr}\in\mathbb F[z]$ with $\deg {f}_{\bbl}<\deg p$ and $\deg {f}_{\bbl}<\deg q$,
find an $f\in\mathbb F[z]$ such that $\deg f<\deg p+\deg q$ and
\begin{equation}
f-f_{\bbl}\in\langle p\rangle_{\bbr} \quad\mbox{and}\quad f-f_{\bbr}\in\langle q\rangle_{\bbl}.
\label{4.25}
\end{equation}}
Indeed, if we let 
\begin{equation}
p=\mathfrak P_{A,\bv}\quad\mbox{and}\quad q=\mathfrak P_{\bu,B}
\label{po14}
\end{equation} 
be minimal polynomials of the pairs $(A,\bv)$ and $(\bu,B)$, and then let 
$$
f_{\bbl}=\mathfrak A_{n}\mathfrak C_{A,\bv}^{-1}{\bf b}\quad\mbox{and}\quad 
f_{\bbr}={\bf d}\mathfrak O_{A,\bv}^{-1}\mathfrak A_{k}^\top, 
$$
then conditions \eqref{4.25} turn out to be identical to parametrization formulas 
\eqref{4.22b}, \eqref{4.22c}, and hence, they are equivalent to conditions \eqref{4.22} and \eqref{4.24}. 
\begin{remark}
{\rm The degree constraint in \eqref{4.25y} is not restrictive. By the left and right division algorithms, 
any polynomial $f\in\mathbb F[z]$ can be uniquely represented as
\begin{equation}
f=\widetilde{f}+phq\quad \mbox{for some} \; \;  \widetilde{f},h\in\mathbb F[z], \; \; \deg\widetilde{f}<\deg p+\deg q,
\label{rep1}
\end{equation}
and furthermore, $f$ satisfies conditions \eqref{4.25} if and only if $\widetilde{f}$ does.}
\label{R:4.3}
\end{remark}
Theorem \ref{T:4.10} below states that the problem \eqref{4.25y} has a solution if 
and only if the Sylvester equation $AY-YB={\bf b}\bu-\bv {\bf d}$ has a solution $Y\in\mathbb F^{n\times k}$.
The ``only if" part of this criterion is verified in the next section via certain ``two-sided" evaluation calculus.

\subsection{Two-sided evaluation} Left and right evaluations \eqref{2.6} based on input
and output pairs evaluate a scalar polynomial $f\in\mathbb F[z]$
at these pairs, rather at matrices (as in \eqref{2.2}). In formula \eqref{4.26} below, we introduce a map
$\mathbb F[z]\to \mathbb F^{n\times k}$ that evaluates $f$ at the couple $\{(A,\bv), (\bu,B)\}$ consisting of an
input pair $(A,\bv)$ and an output pair $(\bu,B)$.
 
\smallskip

Let us extend the backward-shift operators  $L_A$ and $R_B$ defined via formulas \eqref{2.9} on vector polynomials of special form
(polynomial multiples of a constant vector) to matrix polynomials of the form
$$
F=\bv f\bu, \quad\mbox{where}\quad\bv\in\mathbb F^{n\times 1}, \;  \bu\in\mathbb F^{1\times k}, \; f(z)=\sum f_jz^j\in\mathbb F[z],
$$
by the formulas 
\begin{align}
L_A(\bv f\bu)&:=L_A(\bv f)\cdot \bu=\sum_{j+k=0}^{\deg f-1}A^j \bv f_{k+j+1}\bu z^k,\label{2.6n}\\
R_B (\bv f\bu)&:=\bv \cdot R_B (f\bu)=\sum_{j+k=0}^{\deg f-1}\bv f_{k+j+1}\bu B^jz^k.\notag
\end{align}
Evaluating the top formula at $B$ on the right and the bottom formula at $A$ on the left we get the same outcomes which we will refer to as 
the {\em two-sided evaluation} of $f$ at the couple $\{(A,\bv), (\bu,B)\}$:
\begin{align}
(\bv f\bu)^{\bbm}(A,B)&:=\sum_{i+j=0}^{\deg f-1} A^i\bv f_{i+j+1}\bu B^j\label{4.26}\\
&=(L_A (\bu f\bv))^{\br}(B)=(R_B (\bu f\bv))^{\bl}(A).\notag
\label{4.26}
\end{align}
\begin{proposition}
For any $f\in\mathbb F[z]$, the evaluations \eqref{2.6} and \eqref{4.26} satisfy the Sylvester equation
\begin{equation}
A\cdot (\bv f\bu)^{\bbm}(A,B)-(\bv f\bu)^{\bbm}(A,B)\cdot B
=(\bv f)^{\bl}(A)\cdot\bu-\bv\cdot (f\bu)^{\br}(B).
\label{4.27}
\end{equation}
{\rm Indeed, by making use of polynomial expressions in \eqref{2.6} and \eqref{4.27} we get
\begin{align*}
&A\cdot (\bv f\bu)^{\bbm}(A,B)-(\bv f\bu)^{\bbm}(A,B)\cdot B\\
&=\sum_{i+j=0}^{{\rm deg f}-1} A^{i+1}\bv f_{i+j+1}\bu B^j-\sum_{i+j=0}^{{\rm deg f}-1} A^i\bv f_{i+j+1}\bu B^{j+1}\notag\\
&=\sum_{i=1}^{{\rm deg f}}A^{i}\bv f_i\bu -\sum_{j=1}^{{\rm deg f}} \bv f_{j}\bu B^j=
(\bv f)^{\bl}(A)\cdot\bu-\bv\cdot (f\bu)^{\br}(B).
\end{align*}}
\label{R:klj}
\end{proposition}
\begin{corollary}
If a polynomial $f$ satisfies conditions \eqref{4.25y}, then the matrix $Y=(\bv f\bu)^{\bbm}(A,B)$ solves the 
Sylvester equation $AY-YB={\bf b}\bu-\bv {\bf d}$.
\label{C:5.2}
\end{corollary}
For a concrete example, we will compute the two-sided evaluation at $((C_{\boldsymbol\ell}(p), {\bf e}_{1,n}),
({\bf e}_{1,k},C_{\bf r}(q)))$ where $C_{\boldsymbol\ell}(p)$ and $C_{\bf r}(q)$ are companion matrices of given polynomials
$p,q\in\mathbb F[z]$ with $\deg p=n$, $\deg q=k$. We let for short,
\begin{equation}
\Upsilon^{f}=\left[\Upsilon_{ij}\right]:=({\bf e}_{1,n}f {\bf e}_{1,k}^\top)^{\bbm}(C_{\boldsymbol\ell}(p),C_{\bf r}(q))
\label{3.4akl}
\end{equation}
and denote by $\Upsilon_j$ and $\widetilde \Upsilon_i$ the $j$-th column and the $i$-th row of the matrix $\Upsilon^f$.
Upon specifying two last formulas in \eqref{4.26} to the present setting and combining Propositions \ref{P:nul} and \ref{R:nur}
we see that 
\begin{equation}
\Upsilon_j=\big({\bf e}_{1,n}h_j\big)^{\bl}(C_{\boldsymbol\ell}(p))\quad\mbox{and}\quad
\widetilde \Upsilon_i=\big(g_i{\bf e}_{1,k}^\top\big)^{\br}(C_{\bf r}(q)),
\label{3.4akb}
\end{equation}
where the polynomials $g_i$ and $h_j$ are defined in \eqref{2.12x} and \eqref{2.12xx}, respectively.
\begin{remark}
$\Upsilon^{f}=0$ if and only if $f=\alpha+phq$ for some $h\in\mathbb F[z]$
and $\alpha\in\mathbb F$.
\label{R:twos}
\end{remark}
\begin{proof}
Since $q$ is the minimal polynomial of the pair $({\bf e}_{1,k},C_{\bf r}(q))$, it follows from \eqref{3.4akb} that
$\widetilde \Upsilon_i=0$ if and only if $g_i$ is a left multiple of $q$.
Given $f=\alpha+phq$, equalities \eqref{2.15} hold with $b(z)=b_0=\alpha$ and $g=hq$. Hence, $g_i=(R_0^{j}p)\cdot g$ is a left multiple of $q$. 
Therefore, $\widetilde \Upsilon_i=0$ for $i=1,\ldots,k$ and hence, $\Upsilon^f=0$. Conversely, if $\Upsilon^f=0$, then $g_j=r_j q$
for $j=1,\ldots,n$. In particular, $g_n=g=r_n q$, and now it follows from \eqref{2.12x} that $R_0^j b=0$ for $j=1,\ldots,n-1$.
Since $\deg b<n$, the latter equalities imply that $b=b_0$. Then \eqref{2.15} takes the form $f=pr_nq+b_0$ as desired.
\end{proof}
Since any $f\in\mathbb F[z]$ can be represented as in \eqref{rep1} and then necessarily $\Upsilon^{f}=\Upsilon^{\widetilde{f}}$, by
Remark \ref{R:twos}, it follows that it suffices to compute $\Upsilon^{f}$ for $f$ with $\deg f<n+k$.
Any such polynomial can be represented as
\begin{equation}
f=pg+b=hq+d
\label{lkj}
\end{equation}
with $b,h$ of degree less than $n$ and $d,g$ of degree less than $k$, i.e.,
\begin{equation}
\begin{array}{ll}
b(z)=b_0+\ldots+b_{n-1}z^{n-1},\quad &g(z)=g_0+\ldots+g_{k-1}z^{k-1},\\[2mm]
d(z)=d_0+\ldots+d_{k-1}z^{k-1},\quad &h(z)=h_0+\ldots+h_{n-1}z^{n-1}.
\end{array}\label{lkj1}
\end{equation}
Below, we compute the matrix $\Upsilon^f$ in terms of polynomials \eqref{lkj1}.  
\begin{lemma}
If $f$ is of the form \eqref{lkj}, \eqref{lkj1}, then the columns $\Upsilon_j$  and the rows $\widetilde \Upsilon_i$
of the matrix \eqref{3.4akl} are given by
\begin{align}
\Upsilon_k&=\begin{bmatrix}h_0\\ \vdots \\h_{n-1}\end{bmatrix},\quad \widetilde \Upsilon_n=\begin{bmatrix} g_0 & \ldots & g_{k-1}\end{bmatrix},
\label{lkj1a}\\
\Upsilon_j&=\sum_{i=0}^{k-j}C_{\boldsymbol\ell}(p)^i\Upsilon_kq_{j+i}+\sum_{i=0}^{k-j-1}C_{\boldsymbol\ell}(p)^i{\bf e}_{1}d_{j+i}\quad
(j=1,\ldots,k-1),\label{lkj2}\\
\widetilde \Upsilon_i&=\sum_{j=0}^{n-i}p_{i+j}\Upsilon_nC_{\bf r}(q)^j+\sum_{j=0}^{n-i-1}b_{i+j}{\bf e}_1^\top 
C_{\bf r}(q)^j\quad(i=1,\ldots,n-1).\notag
\end{align}
\label{L:syl}
\end{lemma}
\begin{proof}
By \eqref{3.4akb} and \eqref{2.12xx} (for $j=k$) , 
\begin{equation}
\Upsilon_k=\big({\bf e}_{1}h\big)^{\bl}(C_{\boldsymbol\ell}(p)),
\label{lkj4}
\end{equation}
and since $\deg h<\deg p$, the latter column consists of the coefficients of $h$, by Proposition \ref{P:nul}, which verifies the 
first formula in \eqref{lkj1a}. We next compute 
\begin{align*} 
 \Upsilon_j&=\big({\bf e}_{1}h_j\big)^{\bl}(C_{\boldsymbol\ell}(p))\\
&=\big({\bf e}_{1}\big(h\cdot (R_0^j q)+R_0^j d\big)\big)^{\bl}(C_{\boldsymbol\ell}(p))\\
&=\big({\bf e}_{1}\big(h\cdot (R_0^jg\big)\big)^{\bl}(C_{\boldsymbol\ell}(p))+\big({\bf e}_{1}(R_0^j d)\big)^{\bl}(C_{\boldsymbol\ell}(p))\\
&=\big(({\bf e}_{1}h)^{\bl}(C_{\boldsymbol\ell}(p))\cdot R_0^jg\big)^{\bl}(C_{\boldsymbol\ell}(p))
+\big({\bf e}_{1}(R_0^j d)\big)^{\bl}(C_{\boldsymbol\ell}(p))\\
&=\big(\Upsilon_k\cdot R_0^jg\big)^{\bl}(C_{\boldsymbol\ell}(p))+\big({\bf e}_{1}(R_0^j d)\big)^{\bl}(C_{\boldsymbol\ell}(p)),
\end{align*}
where we used \eqref{3.4akb} and \eqref{2.12xx} for the two first steps, the additivity of evaluation operators and the 
multiplicative property \eqref{2.6ups} (applied to $h$ and $R_0^jg$) for the next two step, and finally  we used \eqref{lkj4}
for the last step. The expression on the right side is the same as in \eqref{lkj2}, by definition \eqref{2.6} of left evaluation. Equalities for 
the rows $\widetilde \Upsilon_i$ are verified similarly.
\end{proof}
\subsection{Sylvester equations} We now consider the Sylvester equation
\begin{equation}
C_{\boldsymbol\ell}(p)X-XC_{\bf r}(q)={\bf b} {\bf e}^\top_{1,k}-{\bf e}_{1,n}{\bf d}
\label{lkj1b}
\end{equation}
with given ${\bf b}=\left[b_0 \; \ldots \; b_{n-1}\right]^\top$ and ${\bf d}=\left[d_0 \; \ldots \; d_{k-1}\right]$ and unknown 
$X\in\mathbb F^{n\times k}$, along with two associated ``generalized" Sylvester equations
\begin{align}
\big({\bf x} q\big)^{\bl}(C_{\boldsymbol\ell}(p))&:=
\sum_{i=0}^{k}C_{\boldsymbol\ell}(p)^i{\bf x}q_{i}={\bf b}-\sum_{i=0}^{k-1}C_{\boldsymbol\ell}(p)^i{\bf e}_{1,n}d_{i},\label{pr1}\\
\big(p\widetilde{\bf x}\big)^{\br}(C_{\bf r}(q))&:=\sum_{j=0}^{n}p_{j}\widetilde{\bf x}C_{\bf r}(q)^j={\bf d}-
\sum_{j=0}^{n-1}b_{j}{\bf e}_1^\top C_{\bf r}(q)^j\label{pr2}
\end{align}
with unknowns ${\bf x}\in\mathbb F^{n\times n}$ and $\widetilde{\bf x}\in\mathbb F^{1\times k}$. The next result establishes one-to-one 
correspondences between four sets: solution sets of equations \eqref{lkj1b}, \eqref{pr1}, \eqref{pr2}, and the set of all polynomials $f$
such that 
\begin{equation}
({\bf e}_{1,n}f)^{\bl}(C_{\boldsymbol\ell}(p))={\bf b},\quad (f{\bf e}_{1,k}^\top)^{\br}(C_{\bf r}(q))={\bf d}, \quad
\Upsilon^f=X,\quad \deg f<n+k,
\label{pr5}
\end{equation}
where $\Upsilon^f$ is defined in \eqref{3.4akl}. 
\begin{theorem}
Let $X_j$ and $\widetilde{X}_i$ denote the $j$-th column and the $i$-th row of a matrix $X\in\mathbb F^{n\times k}$. Then

\smallskip

{\rm (1)} If $f\in\mathbb F[z]$ satisfies conditions \eqref{pr5}, then $X=\Upsilon^f$ is a solution to the equation \eqref{lkj1b}.

\smallskip

{\rm (2)} If $X$ solves the equation \eqref{lkj1b}, then
\begin{itemize}
\item[(a)] ${\bf x}=X_k$ solves the equation \eqref{pr1};
\item[(b)] $\widetilde{\bf x}=\widetilde{X}_n$ solves the equation \eqref{pr2};
\item[(c)] The formulas (we recall \eqref{fraka})
\begin{align}
f(z)&={\bf d}\mathfrak A_{k}(z)^\top+\mathfrak A_{n}(z)X{\bf e}_{k,k}q(z)\notag\\
&=\mathfrak A_{n}(z){\bf b}+p(z){\bf e}_{n,n}^\top X\mathfrak A_{k}(z)^\top\label{pr12}
\end{align}
define a unique $f\in\mathbb F[z]$ subject to conditions \eqref{pr5}.
\end{itemize}

{\rm (3)} If ${\bf x}\in\mathbb F^{n\times 1}$ solves the equation \eqref{pr1}, then the formula 
\begin{equation}
f(z)={\bf d}\mathfrak A_{k}(z)^\top+\mathfrak A_{n}(z){\bf x}q(z)
\label{pr10}
\end{equation}
defines a unique $f\in\mathbb F[z]$ subject to conditions \eqref{pr5}, whereas
the columns $X_k={\bf x}$ and
\begin{equation}
X_j=\sum_{i=0}^{k-j}C_{\boldsymbol\ell}(p)^i{\bf x}q_{j+i}+\sum_{i=0}^{k-j-1}C_{\boldsymbol\ell}(p)^i{\bf e}_{1,n}d_{j+i} \quad (1\le j\le k-1),
\label{pr3}
\end{equation}
define the only solution $X\in\mathbb F^{n\times k}$ to equation \eqref{lkj1b} with $X_k={\bf x}$.

\smallskip

{\rm (4)} If $\widetilde{\bf x}$ solves the equation \eqref{pr2}, then the formula 
\begin{equation}
f(z)=\mathfrak A_{n}(z){\bf b}+p(z)\widetilde{\bf x}\mathfrak A_{k}(z)^\top
\label{pr11}
\end{equation}
defines a unique $f\in\mathbb F[z]$ subject to conditions \eqref{pr5}, whereas
the rows $\widetilde X_n=\widetilde{\bf x}$ and
$$
\widetilde{X}_i=\sum_{j=0}^{n-i}p_{i+j}\widetilde{\bf x}C_{\bf r}(q)^j+\sum_{j=0}^{n-i-1}b_{i+j}{\bf e}_1^\top
C_{\bf r}(q)^j\quad(1\le i\le n-1)
$$
define the only solution $X\in\mathbb F^{n\times k}$ to the equation \eqref{lkj1b} with $\widetilde{X}_n=\widetilde{\bf x}$.
\label{L:l1}
\end{theorem}
\begin{proof} Part (1) follows from Corollary \ref{C:5.2} specialized to the present setting.
Making use of the explicit formula 
$$
C_{\bf r}(q)=\begin{bmatrix} 0 & 1 & 0 & \ldots & 0\\
0 & 0 & 1 &\ldots& 0\\
\vdots & \vdots &\vdots & \ddots & \vdots \\
0 & 0& 0 &\ldots &1 \\
-q_0&-q_1&-q_2&\ldots &-q_{k-1} \end{bmatrix},
$$
we equate the corresponding columns in \eqref{lkj1b}:
\begin{align}
C_{\boldsymbol\ell}(p)X_1+X_kq_0&={\bf b}-{\bf e}_{1,n}d_0,\label{lop1}\\
C_{\boldsymbol\ell}(p)X_j-X_{j-1}+X_kq_{j-1}&=-{\bf e}_{1,n}d_{j-1}\quad \mbox{for}\quad j=2,\ldots,k.\label{lop2}
\end{align}
From \eqref{lop2} we recursively recover $X_{k-1}, X_{k-1},\ldots, X_1$ from $X_k$ arriving at formulas \eqref{pr3}
which are the same as \eqref{lkj2}. Substituting the formula \eqref{pr3} (for $j=1$) into \eqref{lop1}, 
and moving all terms not containing $X_k$ to the right side, we get the equality
\begin{equation}
\sum_{i=0}^{k}C_{\boldsymbol\ell}(p)^iX_kq_{i}={\bf b}-\sum_{i=0}^{k-1}C_{\boldsymbol\ell}(p)^i{\bf e}_{1,n}d_{i},
\label{pr6}
\end{equation}
which means that ${\bf x}=X_k$ is a solution to the equation \eqref{pr1}. As is easily seen, the system of equalities
\eqref{pr6}, \eqref{pr3} is equivalent to the system \eqref{lop1}, \eqref{lop2} (i.e., to the Sylvester equality \eqref{lkj1b}).
Therefore, with the fixed column $X_k={\bf x}$ subject to \eqref{pr6}, the only way to extend it to a solution $X$ to the equation
\eqref{lkj1b}, is to use recursive formulas \eqref{pr3}. This completes the proof part (2a) and the second half of part (3). 

\smallskip

To prove the first half, take any ${\bf x}\in\mathbb F^{n\times 1}$ subject to \eqref{pr1} and extend it to the matrix 
$X$ subject to \eqref{lkj1b}
using formulas \eqref{pr3}. Since the formulas \eqref{pr3} are the same as in \eqref{lkj2}, the matrix $X$ equals to the matrix
$\Upsilon^f$ corresponding to the polynomial
\begin{equation}
f=hq+d,\quad\mbox{where}\quad h(z)=\mathfrak A_{n}(z)\cdot {\bf x}\quad\mbox{and}\quad d(z)={\bf d}\cdot\mathfrak A_{k}(z)^\top,
\label{pr7}
\end{equation}
which has been announced in \eqref{pr10}. Thus, this $f$ satisfies the third condition in \eqref{pr5}, and it
is follows from \eqref{pr7} (by Proposition \ref{R:nur}) that it also satisfies the second one, 
and that $\deg f\le \deg h+\deg q<n+k$.
The two last conditions in \eqref{pr5} fix the quotient $h$ and the remainder $d$ of $f$ (of degree less than $n+k$)
when divided by $q$ on the right and hence, determine $f$ uniquely. To show that $f$ of the form \eqref{pr7} also satisfies the first
condition in \eqref{pr5}, let us observe that ${\bf x}$ can be interpreted as the left value
$$
{\bf x}=({\bf e}_{1,n}h)^{\bl}(C_{\boldsymbol\ell}(p)).
$$
Making use of the rule \eqref{ups}, we now can write \eqref{pr1} as
$$
\big({\bf x}q\big)^{\bl}(C_{\boldsymbol\ell}(p))=
\big({\bf e}_{1,n}hq\big)^{\bl}(C_{\boldsymbol\ell}(p))={\bf b}-({\bf e}_{1,n}d)^{\bl}(C_{\boldsymbol\ell}(p)),
$$
which in turn, is equivalent (due to \eqref{pr7}) to
\begin{align*}
{\bf b}&=\big({\bf e}_{1,n}hq\big)^{\bl}(C_{\boldsymbol\ell}(p))+({\bf e}_{1,n}d)^{\bl}(C_{\boldsymbol\ell}(p))\\
&=({\bf e}_{1,n}(hq+d))^{\bl}(C_{\boldsymbol\ell}(p))=({\bf e}_{1,n}f)^{\bl}(C_{\boldsymbol\ell}(p)),
\end{align*}
which completes the proof of part (3). Parts (4) and (2b) are verified similarly. It remains to confirm the part (2a).
To this end, take any $X\in\mathbb F^{n\times k}$ subject to \eqref{lkj1b} and observe that  
\begin{equation}
{\bf x}:=X_k=X{\bf e}_{n,n}\quad\mbox{and}\quad \widetilde{\bf x}:=\widetilde{X}_n={\bf e}_{k,k}X 
\label{pr13}
\end{equation}
solve the respective equations \eqref{pr1} and \eqref{pr2} (by parts (1a) and (1b)) and hence the formulas \eqref{pr10} and \eqref{pr11}
define a unique (and therefore, the same) polynomial $f$ subject to conditions \eqref{pr5}. The formulas in \eqref{pr12} 
follow from \eqref{pr10}, \eqref{pr11} and \eqref{pr13}.
\end{proof}
\subsection{Quasi-ideals in $\mathbb F[z]$} An additive subgroup $\mQ$ of an associative ring $\mathcal A$ such that 
$\mQ\mathcal A\cap\mathcal A\mQ\subseteq \mQ$ (called a {\em quasi-ideal} in \cite{stein2}) amounts, in the setting of 
$\mathbb F[z]$, to the intersection of a left and a right ideal 
$$
\mQ_{p,q}:=\langle p\rangle_{\bf r}\cap\langle q\rangle_{\boldsymbol\ell}
$$
generated by two given polynomials. Any element $f\in \mQ_{p,q}$ is characterized by factorizations $f=pg=hq$ or by homogeneous 
interpolation conditions $({\bf e}_{1,n}f)^{\bl}(C_{\boldsymbol\ell}(p))=0$ and $(f{\bf e}_{1,k}^\top)^{\br}(C_{\bf r}(q))=0$.
By letting ${\bf b}=0$ and ${\bf d}=0$ throughout Section 6.2 we arrive at the following result.
\begin{proposition}
Given polynomials $p,q\in\mathbb F[z]$, the formula
$$
X\mapsto \mathfrak A_{n}(z)X{\bf e}_{k,k}q(z)=p(z){\bf e}_{n,n}^\top X\mathfrak A_{k}(z)^\top
$$
establishes the one-to-one correspondence between the matrices $X\in\mathbb F^{n\times k}$ such that $C_{\boldsymbol\ell}(p)X=XC_{\bf r}(q)$
and the polynomials $f\in\mQ_{p,q}$ of degree $\deg f<\deg p+\deg q$.
\label{R:quasi}
\end{proposition}
By Remark \ref{R:4.3}, any $f\in \mQ_{p,q}$ can be uniquely represented as in \eqref{rep1} with $\widetilde{f}\in \mQ_{p,q}$; hence 
the description of the whole $\mQ_{p,q}$ follows from Proposition \ref{R:quasi}.

\subsection{Two-sided interpolation problems} As an intermediate step toward solving the problem {\bf TSP} \eqref{4.25y}, we will
consider the augmented two-sided problem {\bf ATSP} (equivalent to the problem \eqref{pr5}) whose data set
\begin{equation}
\Omega=\{(A,\bv), \, (\bu,B), \, {\bf b}, \, {\bf d}, \, S\}
\label{data}
\end{equation}
contains, a controllable pair $(A,\bv)$, an observable pair $(\bu,B)$ and the  
target values ${\bf b}\in\mathbb F^{n\times 1}$, ${\bf d}\in\mathbb F^{1\times k}$ and $S\in\mathbb F^{n\times k}$ 
for left, right and two-sided interpolation conditions. The formal definition of the problem is as follows.

\medskip
\noindent
{\bf Problem} ${\bf ATSP}(\Omega)$: {\em Given $\Omega$ as in \eqref{data}, find all $f\in\mathbb F[z]$ subject 
to conditions \eqref{4.25y} and 
\begin{equation}
(\bv f\bu)^{\bbm}(A,B)=S.
\label{4.25ya}
\end{equation}}
By Proposition \ref{P:2.15}, the pairs $(A,\bv)$ and $(\bu,B)$ are similar to $(C_{\boldsymbol\ell}(p),{\bf e}_{1,n})$ and 
$({\bf e}_{1,k}^\top, C_{\bf r}(q))$, respectively; in more detail,
\begin{equation}
\mathfrak C_{A,\bv}^{-1}A\mathfrak C_{A,\bv}=C_{\boldsymbol\ell}(p), \; \; \mathfrak C_{A,\bv}^{-1}\bv={\bf e}_1, \; \; 
\mathfrak O_{\bu,B}B\mathfrak O_{\bu,B}^{-1}=C_{\bf r}(q),\; \; \bu \mathfrak O_{\bu,B}^{-1}={\bf e}_1^\top.
\label{pr14}
\end{equation}
Since similar pairs have the same minimal polynomials, we will use notation \eqref{po14} (i.e., $p=\mathfrak P_{A,\bv}$
and $q=\mathfrak P_{\bu,B}$) throughout this section.
\begin{remark}
Conditions \eqref{4.25y} and \eqref{4.25ya} can be equivalently written as
\begin{align}
({\bf e}_{1,n} f)^{\bl}(C_{\boldsymbol\ell}(p))&={\bf b}':=\mathfrak C_{A,\bv}^{-1}{\bf b},\label{4.24a}\\
(f{\bf e}_{1,k}^\top)^{\br}(C_{\bf r}(q))&={\bf d}':={\bf d}\mathfrak O_{\bu,B}^{-1},\label{4.24aa}\\
({\bf e}_{1,n}f {\bf e}_{1,k}^\top)^{\bbm}(C_{\boldsymbol\ell}(p),C_{\bf r}(q))&=S^\prime:=\mathfrak C_{A,\bv}^{-1}S\mathfrak O_{\bu,B}^{-1}.
\label{4.24aaa}
\end{align}
{\rm Indeed, due to equalities \eqref{pr14},  we have for any $f\in\mathbb F[z]$,
\begin{align*}
({\bf e}_{1,n} f)^{\bl}(C_{\boldsymbol\ell}(p))&=\mathfrak C_{A,\bv}^{-1}\cdot (\bv f)^{\bl}(A),\\
\quad (f{\bf e}_{1,k}^\top)^{\br}(C_{\bf r}(q))&=(f\bu)^{\br}(B)\cdot \mathfrak O_{\bu,B}^{-1},
\end{align*}
by Remark \ref{R:1.2}, while for the two-sided evaluation, we have from \eqref{4.26},
$$
({\bf e}_{1,n} f{\bf e}_{1,k}^\top)^{\bbm}(C_{\boldsymbol\ell}(p),C_{\bf r}(q))=\mathfrak C_{A,\bv}^{-1}\cdot(\bv f\bu)^{\bbm}(A,B)\cdot\mathfrak O_{\bu,B}^{-1}.
$$
Now we see from the latter equalities that interpolation conditions \eqref{4.24a}--\eqref{4.24aaa} are obtained from \eqref{4.25y}, \eqref{4.25ya}
upon multiplying the latter by invertible $T$ and  $\widetilde{T}$  on the left and/or on the right respectively, and hence, the
asserted equivalence follows.}
\label{R:m2}
\end{remark}
As we know from Theorem \ref{L:l1}, the interpolation problem \eqref{4.24a}--\eqref{4.24aaa}) has a solution if and only if 
\begin{equation}
C_{\boldsymbol\ell}(p)S^\prime-S^\prime C_{\bf r}(q)={\bf b}'{\bf e}_{1,k}^\top -{\bf e}_{1,n}{\bf d}',
\label{pr13a}
\end{equation}
in which case the only solution is given by \eqref{pr12}, i.e., 
\begin{equation}
f=\mathfrak A_{n}{\bf b}'+p{\bf e}_{n,n}^\top S'\mathfrak A_{k}^\top=
{\bf d}^\prime\mathfrak A_{k}^\top+\mathfrak A_{n}S'{\bf e}_{k,k}q.
\label{pr14a}
\end{equation}
Using relations \eqref{pr14} and the rightmost equalities in \eqref{4.24a}--\eqref{4.24aaa}, we may write \eqref{pr13a} and \eqref{pr14a}
in terms of $\Omega$ to arrive at the following result. 
\begin{theorem} The problem {\bf ATSP} has a solution if and only if 
\begin{equation}
AS-SB={\bf b}\bu-\bv {\bf d}, 
\label{pr15}
\end{equation}
in which case the only solution is given by either formula 
\begin{align}
f(z)&=\mathfrak A_{n}(z)\mathfrak C_{A,\bv}^{-1}{\bf b}+
\mathfrak P_{A,\bv}(z){\bf e}_{n,n}^\top \mathfrak C_{A,\bv}^{-1}S\mathfrak O_{\bu,B}^{-1}\mathfrak A_{k}(z)^\top\notag\\
&={\bf d}\mathfrak O_{\bu,B}^{-1}\mathfrak A_{k}(z)^\top
+\mathfrak A_{n}(z)\mathfrak C_{A,\bv}^{-1}S\mathfrak O_{\bu,B}^{-1}{\bf e}_{k,k}\mathfrak P_{\bu,B}(z).
\label{4.29}
\end{align}
\label{T:4.10u}
\end{theorem}
Note that since the first terms in formulas \eqref{4.29} solve the respective one-sided problems \eqref{4.22} and \eqref{4.24} and the second terms are 
multiples of $\mathfrak P_{A,\bv}$ and $\mathfrak P_{\bu,B}$, respectively, it is immediate that $f$ defined in \eqref{4.29} satisfies conditions \eqref{4.25y}. A nontrivial part here is 
that $f$ also satisfies the third condition \eqref{4.25ya} and that two formulas in \eqref{4.29} represent the same polynomial. In particular, it follows
from \eqref{4.29} that the problem {\bf ATSP}($\Omega$) is redundant: if \eqref{pr15} holds and $f$ satisfies \eqref{4.25ya} and any one of the two 
conditions in \eqref{4.25y}, then it also satisfies the second. In fact, the condition \eqref{4.25ya} alone determines the polynomial 
$f$ of degree less than $n+k$ up to a constant. 
\begin{theorem}
Given $\Omega=\{(A,\bv), \, (\bu,B), \, S\}$ as above, there is  a polynomial $f\in\mathbb F[z]$ subject to condition
\eqref{4.24aaa} if and only if
\begin{equation}
({\bf I}_n-{\bf e}_{1,n}{\bf e}^\top_{1,n})\mathfrak C_{A,\bv}^{-1}(AS-SB)\mathfrak O_{\bu,B}^{-1}({\bf I}_k-{\bf e}_{1,k}{\bf e}^\top_{1,k})=0,
\label{pr16}
\end{equation}
in which case a solution is uniquely defined (up to an arbitrary additive constant $\alpha\in\mathbb F$) by the formula
\begin{align}
f(z)=&\alpha +\mathfrak A_{n}(z)\mathfrak C_{A,\bv}^{-1}(AS-SB)\mathfrak O_{\bu,B}^{-1}{\bf e}_{1,n}\notag\\
&\quad +\mathfrak P_{A,\bv}(z){\bf e}_{n,n}^\top \mathfrak C_{A,\bv}^{-1}S\mathfrak O_{\bu,B}^{-1}\mathfrak A_{k}(z)^\top.
\label{pr17}
\end{align}
\label{T:lmn}
\end{theorem}
\begin{proof}
As in the proof of the previous theorem, we pass to the equivalent interpolation problem \eqref{4.24aaa} (with $p$ and $q$ as in \eqref{po14}). 
If there is a polynomial $f\in\mathbb F[z]$ satisfying
\eqref{4.24aaa}, then equality \eqref{pr13a} holds for some ${\bf b}'$ and ${\bf d}'$. Then we also have 
\begin{equation}
({\bf I}_n-{\bf e}_{1,n}{\bf e}^\top_{1,n})(C_{\boldsymbol\ell}(p)S'-S'C_{\bf r}(q))({\bf I}_k-{\bf e}_{1,k}{\bf e}^\top_{1,k})=0,
\label{pr18}
\end{equation}
which is the same as \eqref{pr16}, due to relations \eqref{pr14} and the rightmost equality in \eqref{4.24aaa}. Conversely, 
if \eqref{pr16} holds, we see from \eqref{pr18} (which is equivalent to \eqref{pr16}) that equality \eqref{pr13a} holds for 
\begin{align}
{\bf b}^\prime&={\bf e}_{1,n}\alpha+(C_{\boldsymbol\ell}(p)S'-S'C_{\bf r}(q)){\bf e}_{1,k}, \label{pr19}\\
{\bf d}'&=\alpha{\bf e}_{1,k}^\top-{\bf e}^\top_{1,n}(C_{\boldsymbol\ell}(p)S'-S'C_{\bf r}(q))({\bf I}_k-{\bf e}_{1,k}{\bf e}^\top_{1,k}),
\quad \alpha\in\mathbb F,\notag
\end{align}
and that conversely, if \eqref{pr13a} holds for some ${\bf b}'$ and ${\bf d}'$, the latter two are necessarily of the form \eqref{pr19}
for some $\alpha\in\mathbb F$ (the formulas \eqref{pr19} can be made more symmetric upon shifting the parameter $\alpha$ but we do not 
need this). Now we use the first formula in \eqref{pr14a} with ${\bf b}'$ as in \eqref{pr19} to get
$$
f(z)=\mathfrak A_{n}(z){\bf e}_{1,n}\alpha+
\mathfrak A_{n}(z)(C_{\boldsymbol\ell}(p)S'-S'C_{\bf r}(q)){\bf e}_{1,k}+p(z){\bf e}_{n,n}^\top S'\mathfrak A_{k}(z)^\top.
$$
Replacing in the latter formula $C_{\boldsymbol\ell}(p)$, $C_{\bf r}(q)$, $S'$ by $A$, $B$, $S$ according to \eqref{pr14}, \eqref{4.24aaa} 
and taking into account that $\mathfrak A_{n}(z){\bf e}_{1,n}=1$, we get \eqref{pr17}.
\end{proof}
Now we drop the two-sided condition \eqref{4.25ya} getting back to the problem {\bf TSP}, namely: {\em given $\Omega$ as in \eqref{data} 
(without $S$ though), find an $f\in\mathbb F[z]$ subject to interpolation conditions \eqref{4.25y}}.
\begin{theorem}
The problem {\bf TSP} has a solution if and only if the Sylvester equation
\begin{equation}
AY-YB={\bf b}\bu-\bv {\bf d}
\label{4.28}
\end{equation}
admits a solution $Y\in\mathbb F^{n\times k}$. For each such solution $Y$, the polynomial
\begin{align}
f_Y(z)&=\mathfrak A_{n}(z)\mathfrak C_{A,\bv}^{-1}{\bf b}+\mathfrak P_{A,\bv}(z){\bf e}_{n,n}^\top 
\mathfrak C_{A,\bv}^{-1}Y\mathfrak O_{\bu,B}^{-1}\mathfrak A_{k}(z)^\top\notag \\
&={\bf d}\mathfrak O_{\bu,B}^{-1}\mathfrak A_{k}(z)^\top
+\mathfrak A_{n}(z)\mathfrak C_{A,\bv}^{-1}Y\mathfrak O_{\bu,B}^{-1}{\bf e}_{k,k}\mathfrak P_{\bu,B}(z)
\label{4.38}
\end{align}
satisfies conditions \eqref{4.25y}. Moreover either of the formulas \eqref{4.38} establishes a one-to-one 
correspondence between solutions $Y$ to the equation \eqref{4.28} and solutions to the problem  
{\bf TSP}.
\label{T:4.10}
\end{theorem}
\begin{proof}
As in the previous proof, we pass to the equivalent interpolation problem with interpolation conditions 
\eqref{4.24a}, \eqref{4.24aa}. We next multiply both  sides of \eqref{4.28} by $\mathfrak C_{A,\bv}^{-1}$ on the left and 
by $\mathfrak O_{\bu,B}^{-1}$ on the right. On account of \eqref{pr14} and the rightmost definitions in 
\eqref{4.24a}, \eqref{4.24aa}, the resulting equality can be written as
\begin{equation}
C_{\boldsymbol\ell}(p)X-XC_{\bf r}(q)={\bf b}'{\bf e}_{1,k}^\top-{\bf e}_{1,n}{\bf d}',\quad\mbox{where}\quad
X=\mathfrak C_{A,\bv}^{-1}Y\mathfrak O_{\bu,B}^{-1}.
\label{4.28a}
\end{equation}
Since $Y$ solves the Sylvester equation \eqref{4.28} if and only if $X=\mathfrak C_{A,\bv}^{-1}Y\mathfrak O_{\bu,B}^{-1}$
solves \eqref{4.28a}, all the statements now follow from Theorem \ref{L:l1}. The formulas for 
$f_X$ are the same as in \eqref{pr14a} (but with $X$ instead of $S'$). Writing these formulas in terms of $\Omega$ and $Y$
(rather than $\Omega'$ and $X$), again making use of  \eqref{pr14} and the rightmost equalities in \eqref{4.24a}, \eqref{4.24aa},
we get \eqref{4.38}.
\end{proof}
\begin{remark}
{\rm Note that the actual parameters in formulas \eqref{4.38} are the bottom row and the rightmost column of the matrix
$\mathfrak C_{A,\bv}^{-1}Y\mathfrak O_{\bu,B}^{-1}$ rather than the whole matrix $Y$. In other words}, the number of
independent scalar parameters in the parametrization formulas \eqref{4.38} is at most $\min\{n,k\}$.
\label{R:feb15}
\end{remark}
\begin{remark}
The polynomial $f_Y$ defined in \eqref{4.38} also can be written as 
\begin{align}
f_Y(z)&=\mathfrak A_{n}(z)\mathfrak C_{A,\bv}^{-1}\left({\bf b}+(z{\bf I}_n-A)Y
\mathfrak O_{\bu,B}^{-1}\mathfrak A_{k}(z)^\top\right)\notag \\
&=\left( {\bf d}
+\mathfrak A_{n}(z)\mathfrak C_{A,\bv}^{-1}Y(z{\bf I}_k-B)\right)\mathfrak O_{\bu,B}^{-1}\mathfrak P_{\bu,B}(z).
\label{4.38u}
\end{align}
\label{R:alt}
\end{remark}
\begin{proof}
We first observe that the matrices $\mathfrak C_{A,\bv}$ and $\mathfrak O_{\bu,B}$ satisfy equalities
\begin{equation}
A\mathfrak C_{A,\bv}-\mathfrak C_{A,\bv}F_n=A^n\bv{\bf e}_{n,n}^\top,\quad \mathfrak O_{\bu,B}B-F_k^\top\mathfrak O_{\bu,B}={\bf e}_{k,k}\bu B^k.
\label{4.28b}
\end{equation}
where $F_n=\mathcal J_n(0)$ is the $n\times n$ lower triangular Jordan block with zeros on the main diagonal. Indeed, by the definition \eqref{1.2x}
of $\mathfrak C_{A,\bv}$, we have 
\begin{align*}
A\mathfrak C_{A,\bv}-\mathfrak C_{A,\bv}F_n&=\begin{bmatrix} A\bv & \ldots & A^{n-1}\bv & A^{n}\bv\end{bmatrix}
-\begin{bmatrix} A\bv & \ldots & A^{n-1}\bv & 0\end{bmatrix}\\
&=\begin{bmatrix} 0 & \ldots & 0 & A^{n}\bv\end{bmatrix}=A^n\bv{\bf e}_{n,n}^\top,
\end{align*}
verifying the first equality in \eqref{4.28b}. The second follows similarly from the definition \eqref{1.2u}. 
Making use of equalities \eqref{4.28b} along with the identity 
$$
z^n{\bf e}_{n,n}^\top+\mathfrak A_{n}(z)F_n=z\mathfrak A_{n}(z)
$$
(see \eqref{fraka}) and explicit formulas \eqref{1.4} and \eqref{1.4u} of $\mathfrak P_{A,\bv}$ and $\mathfrak P_{\bu,B}$, we get
\begin{align}
\mathfrak P_{A,\bv}(z){\bf e}_{n,n}^\top\mathfrak C_{A,\bv}^{-1}&=
z^n{\bf e}_{n,n}^\top\mathfrak C_{A,\bv}^{-1}-\mathfrak A_{n}(z)\mathfrak C_{A,\bv}^{-1}A^n\bv{\bf e}_{n,n}^\top\mathfrak C_{A,\bv}^{-1}\notag \\
&=z^n{\bf e}_{n,n}^\top\mathfrak C_{A,\bv}^{-1}-\mathfrak A_{n}(z)\mathfrak C_{A,\bv}^{-1}\big(A\mathfrak C_{A,\bv}-\mathfrak C_{A,\bv}F_n\big)
\mathfrak C_{A,\bv}^{-1}\notag \\
&=\big(z^n{\bf e}_{n,n}^\top+\mathfrak A_{n}(z)F_n\big)\mathfrak C_{A,\bv}^{-1}-\mathfrak A_{n}(z)\mathfrak C_{A,\bv}^{-1}A\notag \\
&=z\mathfrak A_{n}(z)\mathfrak C_{A,\bv}^{-1}-\mathfrak A_{n}(z)\mathfrak C_{A,\bv}^{-1}A\notag \\
&=\mathfrak A_{n}(z)\mathfrak C_{A,\bv}^{-1}\big(z{\bf I}_n-A\big).\notag
\end{align}
and similarly,
$$
\mathfrak O_{\bu,B}^{-1}{\bf e}_{k,k}\mathfrak P_{\bu,B}(z)=
\left(z{\bf I}_k-B\right)\mathfrak O_{\bu,B}^{-1}\mathfrak A_{k}(z)^\top. 
$$
Substituting the two latter equalities into \eqref{4.38}, we arrive at \eqref{4.38u}.
\end{proof}
As an application of Theorem \ref{T:4.10}, we get simple sufficient conditions for 
the problem {\bf TSP} to have a unique solution. 
\begin{proposition}
Given the data set \eqref{data}, let us assume that $A$ is algebraic and let
$\bmu_A(z)=z^\kappa+\mu_{\kappa-1}z^{\kappa-1}+\ldots+\mu_0$  be its minimal central polynomial. If
the  matrix $\bmu_{A}(B)$ is invertible, then the problem {\bf TSP} has a unique solution 
given by formulas \eqref{4.38} with 
\begin{equation}
Y=\sum_{i=1}^{\kappa}\mu_{j}\sum_{i=0}^{j-1}A^i(\bv {\bf d}-{\bf b}\bu)B^{j-i-1} \cdot \bmu_{A}(B)^{-1}.
\label{lowdeg}
\end{equation}
\label{P:ex1}
\end{proposition}
\begin{proof}
Since $\bmu_A\in Z_{\mathbb F}[z]$ and $\bmu_{A}(A)=0$,
we have for any $Y\in\mathbb F^{n\times k}$,
\begin{equation}
-Y\bmu_{A}(B)=\bmu_{A}(A)Y -Y\bmu_{A}(B)=\sum_{j=1}^{\kappa}\mu_{j}\sum_{i=0}^{j-1}A^i(AY-YB)B^{j-i-1}.\label{po7}
\end{equation}
If $Y$ satisfies \eqref{4.28}, we replace $AY-YB$ on the right side of \eqref{po7} by ${\bf b}\bu-\bv {\bf d}$ and see that
$Y$ is uniquely defined from \eqref{po7} by the formula \eqref{lowdeg}.
To verify that $Y$ of the form \eqref{lowdeg} indeed satisfies
\eqref{4.28}, we use equality \eqref{po7} with $\bv {\bf d}-{\bf b}\bu$ instead of $Y$:
$$
({\bf b}\bu-\bv {\bf d})\bmu_{A}(B)=\sum_{j=1}^{\kappa}\mu_{j}\sum_{i=0}^{j-1}A^i(A(\bv {\bf d}-{\bf b}\bu)-(\bv {\bf d}-{\bf b}\bu)B)B^{j-i-1}.
$$
If $Y$ is defined as in \eqref{lowdeg}, the expression on the right side 
can be written as $AY\bmu_{A}(B)-Y\bmu_{A}(B)B$. Since the matrices $B$ and $\bmu_A(B)$ commute, we therefore, have
$$
({\bf b}\bu-\bv {\bf d})\bmu_{A}(B)=AY\bmu_{A}(B)-YB\bmu_{A}(B),
$$
which is equivalent to \eqref{4.28}, since $\bmu_{A}(B)$ is invertible.
The rest follows by Theorem \ref{T:4.10}.
\end{proof}
\subsection{Lagrange interpolation} Given interpolation nodes $\alpha_1,\ldots,\alpha_n$ and $\beta_1,\ldots,\beta_k$ in $\mathbb F$
along with target values $b_1,\ldots,b_n$, $d_1,\ldots, d_k$, the {\em two-sided Lagrange interpolation problem} consists of finding an
$f\in\mathbb F[z]$ such that
\begin{equation}
f^{\bl}(\alpha_i)=b_i\quad (i=1,\ldots,n)\quad\mbox{and}\quad f^{\br}(\beta_j)=d_j\quad (j=1,\ldots,k).
\label{5.1}
\end{equation}
We refer to \cite{bollag} for a detailed treatment of this problem. Here we only show that under the assumption that
\begin{equation}
\begin{array}{l}
\mbox{(a) the set $\Lambda_{\boldsymbol\ell}=\{\alpha_1,\ldots,\alpha_n\}$ is left $P$-independent,}\\
\mbox{(b) the set $\Lambda_{\bf r}=\{\beta_1,\ldots,\beta_k\}$ is right $P$-independent,}
\end{array}
\label{ass}
\end{equation}
the problem can be embedded into the scheme of ${\bf TSP}$. To this end, note 
that interpolation conditions \eqref{4.22} and \eqref{4.24} specified to the case
\begin{equation}
\begin{array}{rlll}
A=\sbm{ \alpha_1 & & 0 \\ &\ddots & \\ 0 & & \alpha_n},&\; B=\sbm{\beta_1 & & 0 \\ &\ddots & \\ 0 & & \beta_k},&
{\bf v}=\sbm{ 1 \vspace{-1mm}\\  \vdots \vspace{1mm}\\ 1},&
{\bf b}=\sbm{ b_1 \\ \vdots \\ b_n}, \vspace{1mm}\\
{\bf u}=\begin{bmatrix} 1 & \ldots & 1\end{bmatrix},&\; {\bf d}=\begin{bmatrix} d_1 & \ldots & d_k\end{bmatrix}
&&\end{array}
\label{4.90}
\end{equation}
amount to conditions \eqref{5.1}. 
By Proposition \ref{P:br4}, the assumptions \eqref{ass} ensure the pair $(A,\bv)$ be controllable 
and the pair $(\bu,B)$ be observable and hence, all general results from Section 4.5 apply.
Theorem \ref{T:4.10} describes all solutions 
to the problem \eqref{5.1} in terms of solutions $Y=\left[y_{ij}\right]$ of the Sylvester equation \eqref{4.28},
which in the present setting breaks up into the system of 
$nk$ scalar equations 
\begin{equation}
\alpha_i y_{ij}-y_{ij}\beta_j=b_i-d_j\qquad (1\le i\le n, \; 1\le j\le k).
\label{po5}
\end{equation}
In the case where $\alpha_i$ or $\beta_j$ are algebraic over $Z_{\mathbb F}$, the solvability
criterion for the equation \eqref{po5} (as well as the parametrization of all solutions in the indeterminate case) known
from \cite{jacob1} lead to an explicit description of all solutions to the problem \eqref{5.1}. 

\smallskip 

{\bf Acknowledgements:} The project was partially supported by by Simons Foundation grant 524539

\smallskip

{\bf Declaration of competing interest:}
The author declared that he had no conflicts of interest with respect to their authorship or the publication of this article.

\bibliographystyle{amsplain}

\begin{thebibliography}{00}


\bibitem{bir}
G.~Birkhoff, {\em The algebra of multivariate interpolation}, in:  {\em Constructive approaches to
mathematical models}, pp. 345--363, Academic Press, Ont., 1979.






\bibitem{bollag}
V.~Bolotnikov, {\em Lagrange interpolation over division rings}, Comm. Algebra {\bf 48} (2020), no. 9, 4065--4084.


\bibitem{cohn}
P.~M.~Cohn, {\em Free rings and their relations}, Academic Press, London, 1971.

\bibitem{cohn1}
P.~M.~Cohn, {\em The similarity reduction of matrices over a skew field}, Math. Z. {\bf 132} (1973) 151--163.


\bibitem{cohn3}
P.~M.~Cohn, {\em Skew fields. Theory of general division rings}, Encyclopedia of Mathematics and its Applications {\bf 57},
Cambridge University Press, Cambridge, 1995.


\bibitem{fit}
H.~Fitting, {\em \"Uber den Zusammenhang zwischen dem Begriff der Gleichartigkeit zweier Ideale und dem \"Aquivalenzbegriff der
Elementarteilertheorie}, Math. Ann. {\bf 112} (1936), no. 1, 572-–582.




\bibitem{gm}
B.~Gordon and T.~S.~Motzkin. {\em On the zeros of polynomials over division rings},
Trans. Amer. Math. Soc., {\bf 116} (1965) 218--226,


\bibitem{jacob}
N.~Jacobson, {\em The Theory of Rings}, American Mathematical Society, New York, 1943.


\bibitem{jacob1}
N.~Jacobson, {\em The equation $x'\equiv xd-dx=b$}, Bull. Amer. Math. Soc. {\bf 50}, (1944). 902–-905.



\bibitem{kalman}
R.~E.~Kalman, {\em Contributions to the theory of optimal control}, Bol. Soc. Mat. Mexicana {\bf 5} (1960), 102--119. 

\bibitem{lam1}
T.~Y.~Lam, {\em A general theory of Vandermonde matrices}, Exposition. Math. {\bf 4} (1986), no. 3,
193--215.


\bibitem{ll1}
T.~Y.~Lam and A.~Leroy, {\em Vandermonde and Wronskian matrices over division rings}, J. Algebra {\bf 119} (1988), no. 2,
308-–336.

\bibitem{lamler1}
T.~Y.~Lam and A.~Leroy, {\em Algebraic conjugacy classes and skew polynomial rings}, in {\em Perspectives in ring theory}, pp. 153-–203,
NATO Adv. Sci. Inst. Ser. C Math. Phys. Sci., {\bf 233}, Kluwer Acad. Publ., Dordrecht, 1988.



\bibitem{ll2}
T.~Y.~Lam and A.~Leroy, {\em Wedderburn polynomials over division rings. I}, J. Pure Appl. Algebra {\bf 186} (2004),
no. 1, 43–-76.




\bibitem{llo}
T.~Y.~Lam, A.~Leroy and A.~Ozturk, {\em A. Wedderburn polynomials over division rings. II},
in : {\em Noncommutative rings, group rings, diagram algebras and their applications}, pp. 73-–98,
Contemp. Math. {\bf 456}, Amer. Math. Soc., Providence, RI, 2008.



\bibitem{ore}
O.~Ore, {\em Theory of non-commutative polynomials}, Ann. of Math. {\bf 34} (1933), no. 3, 480--508.

\bibitem{salom}
L.~Solomon, {\em Similarity of the companion matrix and its transpose.
With an appendix by Robert M. Guralnick}, Linear Algebra Appl. {\bf 302/303} (1999),
555–-561.

\bibitem{stein2}
O.~Steinfeld, {\em Quasi-ideals in rings and semigroups},
Hung. Math. Investigations, {\bf 10}. Akad\'emiai Kiad\'o, Budapest, 1978.


\bibitem{wieg}
N.Wiegmann, {\em Some theorems on matrices with real quaternion elements},
Canad. J. Math. {\bf 7} (1955) 191--201.

\bibitem{wed}
J.~H.~M.~Wedderburn, {\em Lectures on Matrices}, Am. Math. Soc., Colloq. Publ. {\bf 17}. Providence,
RI, 1934.


\end{thebibliography}

\end{document}